\documentclass[hidelinks,onefignum,onetabnum]{siamart250211}

\usepackage{lipsum}
\usepackage{graphicx}
\usepackage{epstopdf}
\ifpdf
  \DeclareGraphicsExtensions{.eps,.pdf,.png,.jpg}
\else
  \DeclareGraphicsExtensions{.eps}
\fi

\newsiamremark{remark}{Remark}
\newsiamremark{hypothesis}{Hypothesis}
\crefname{hypothesis}{Hypothesis}{Hypotheses}
\newsiamthm{claim}{Claim}
\newsiamremark{fact}{Fact}
\crefname{fact}{Fact}{Facts}
\newsiamthm{assumption}{Assumption}

\headers{ Accelerated Quasar-Convex Optimization}{E. Gbaguidi and J. Hermant}

\title{Accelerated Stochastic Zeroth-Order Quasar-Convex Optimization \thanks{Submitted to the editors DATE.
\funding{This project has benefited from state support managed by the Agence Nationale de la Recherche (French National Research Agency) under the reference ANR-20-SFRI-0001.}}}

\author{Eméric Gbaguidi\thanks{Equal contributions. Institut de Mathématiques de Bordeaux, Université de Bordeaux
  (\email{thierry-emeric.gbaguidi@math.u-bordeaux.fr}\thanks{Equal contributions},\email{julien.hermant@math.u-bordeaux.fr}).}
\and Julien Hermant\footnotemark[2]}

\usepackage{amsopn}

\usepackage{amsmath,amsfonts,amssymb,mathtools}
\usepackage{mathrsfs} 

\usepackage{enumitem}   
\usepackage{algorithmic}
\usepackage{algorithm}

\usepackage{xcolor,dsfont,todonotes,ulem}
\usepackage{thmtools}
\usepackage{thm-restate}

\newcommand{\EG}[1]{{\color{blue} #1}}

\def\xiest{{\xi^{\{m\}}}}
\def\xiestalg{{\xi_k^{\{m\}}}}

\def\vest{{v}}
\def\vestalg{{v_k}}
\newcommand{\vestf}[1]{v}
\def\gest{{g_\alpha}}
\newcommand{\ceil}[1]{\left\lceil {#1} \right\rceil}
\def\R{{\mathbb{R}}}

\def\N{{\mathbb{N}}}

\def\sphere{\mathbb{S}}
\def\ball{\mathbb{B}}
\def\1{\mathbb{1}}

\def\xz{(x_t,z_t)_{t\ge 0}}

\def\mart{(M_t)_{t\ge 0}}
\def\tk{\{ T_k \}_{k\in \N}}
\newcommand{\I}{\mathbf{I}}
\newcommand{\proxf}{h}
\newcommand{\xalg}[1]{\tilde{x}_{#1}}
\newcommand{\yalg}[1]{\tilde{y}_{#1}}
\newcommand{\zalg}[1]{\tilde{z}_{#1}}
\newcommand{\xalgt}[1]{\tilde{x}_{#1}}

\renewcommand{\P}{\mathbb{P}}
\newcommand{\bpar}[1]{\left(#1\right)}
\newcommand{\prox}[2]{\text{Prox}_{#2}\left(#1\right)}

\newcommand{\E}[1]{\mathbb{E}\left[#1 \right]}

\newcommand{\normgen}[1]{\left\lVert#1\right\rVert}
\newcommand{\norm}[1]{\left\lVert#1\right\rVert_{2}}
\newcommand{\normp}[1]{\left\lVert#1\right\rVert_{p}}
\newcommand{\normq}[1]{\left\lVert#1\right\rVert_{q}}
\newcommand{\Ec}[2]{\mathbb{E}_{#2}\left[#1 \right]}
\newcommand{\Ecbig}[2]{\mathbb{E}_{#2}\Bigg[#1 \Bigg]}
\newcommand{\dotprod}[1]{\left< #1\right>}

\newcommand{\off}[1]{}

\newcommand{\bigO}{\mathcal{O}}

\def\argmin{\textup{argmin}\,}

\ifpdf
\hypersetup{
  pdftitle={ Accelerated Quasar-Convex Optimization},
  pdfauthor={E. Gbaguidi, and J. Hermant}
}
\fi

\begin{document}

\maketitle

\begin{abstract}
We consider unconstrained minimization of smooth quasar-convex functions when only noisy function evaluations are accessible through a stochastic zeroth-order oracle.
For these non-convex functions, the standard acceleration method relies on subspace-search mechanisms that require first-order information, being therefore unavailable in zeroth-order regimes. In contrast, the alternative and less conventional continuized method enable to avoid such mechanisms. In this work, we design a zeroth-order continuized algorithm, leading to accelerated convergence guarantees that parallel those of smooth convex optimization up to a quasar-convexity parameter. 
Our method incorporates a mirror step, improving the dimension dependence when there exists sparse solution.
\end{abstract}

\begin{keywords}
Non-convex optimization, zeroth-order optimization, stochastic optimization
\end{keywords}

\begin{MSCcodes}
90C15, 90C26, 90C56
\end{MSCcodes}

\section{Introduction}
 We consider unconstrained minimization of a smooth function $f:\R^d\to\R$ that writes as follows
\begin{equation}\label{prob:intro}
    \min_{x \in \R^d}~ f(x), \quad \text{where for all }  x\in \R^d,~  f(x) :=  \Ec{f(x,\xi)}{\xi}, 
\end{equation}
where $\xi$ is a random variable belonging to a space $\Xi$, following a distribution $\mathcal{P}_\xi$. A typical example encompassed by \eqref{prob:intro} is empirical risk minimization, where $\mathcal{P}_\xi$ is some data distribution.
Assuming that $\inf_{x\in\R^d} f(x)>-\infty$, and for a prescribed accuracy
$\varepsilon>0$, our goal is to compute an $\varepsilon$-optimal solution
$\tilde x$ that satisfies
\begin{equation}\label{eq:prob}
    f(\tilde x)-\inf_x f(x) \le \varepsilon.
\end{equation}
A wide range of optimization methods have been proposed to address this problem, differing in the type of information they require.
In this work, we focus on zeroth-order methods, which access stochastic evaluations of the function. If fast differentiation techniques have made that in many situations the computation of the full gradient is comparable to that of evaluating
the function itself \cite{griewank2008evaluating}, there remain important settings in which gradient evaluations are either prohibitively expensive or simply unavailable. Examples include bandit optimization \cite{agarwal2010optimal,bartlett2008high,flaxman2004online}, reinforcement learning \cite{choromanski2018structured} or more recently the fine-tuning of large language models \cite{gautam2024variance,malladi2023fine}.
In this work, we are specifically interested in solving \eqref{eq:prob} with fewer function evaluations as possible.

\noindent
\textbf{Acceleration with momentum for convex functions \:}  
Two evaluations of the function can be used to build finite-difference approximations of the gradient, for example
\begin{equation}\label{eq:intro:estimator}
    \frac{f(x+\alpha v)-f(x)}{\alpha}v,
\end{equation}
for some random direction vector $v \in \R^d$ and smoothing parameter $\alpha > 0$. With this perspective, it is natural to wonder if mechanisms that accelerate optimization when using first-order algorithms, i.e. algorithms that access the gradient, can also accelerate when using such approximations. In the case of smooth and convex deterministic first-order optimization, it is well established that using Nesterov's momentum not only improves over gradient descent convergence speed \cite{nesterov1983method,nesterov2004introductory}, but is also worst case optimal \cite{nemirovskij1983problem}. Precisely, it achieves an $\varepsilon$-optimal solution \eqref{eq:prob} in $\bigO(1/\sqrt{\varepsilon})$ iterations, compared to the $\bigO(1/\varepsilon)$ rate of standard gradient descent. Replacing the gradient by \eqref{eq:intro:estimator}, as long as $\alpha$ is chosen small enough and with $v$ chosen at random following a suitable distribution, one can recover similar complexity rate up to the dimension $d$ that appears in the factor, namely $\bigO(d/\sqrt{\varepsilon})$ for zeroth-order Nesterov momentum and $\bigO(d/\varepsilon)$ for zeroth-order gradient descent \cite{gorbunov2018accelerated,nesterov2017random}.
But if the convexity assumption is  convenient for theoretical analysis, there is a large body of problems that do not fit in \cite{bhojanapalli2016global,dauphin2014identifying,ge2017no}. A valuable question is whether and to what extent momentum methods can be effective beyond the convex
setting.

\noindent
\textbf{Acceleration for non-convex functions \:}
In the case of first-order optimization, there exist settings in which momentum cannot improve upon the convergence rate
of gradient descent.
Notable examples include general $L$-smooth functions \cite{lowerboundI} and
$L$-smooth functions satisfying the Polyak--\L{}ojasiewicz condition
\cite{PLlowerbound}. Zeroth-order optimization has been studied in these nonconvex settings, see \textit{e.g.} \cite{Chen2019ZOAdaMMZA,farzin2024minimisation,ghadimi2013stochastic,huang2022accelerated,Ji2019ImprovedZV,Liu2019signSGDVZ,Shi2024GradientFreeMF,Tang2019DistributedZA}, but if momentum can help on specific aspects--such as when using variance reduction \cite{huang2022accelerated}--we do not expect acceleration using momentum, as it cannot occur even with full gradient information.
In contrast, theoretical acceleration for gradient-based algorithms using momentum is known to be achievable for more
structured non-convex classes, such as \textit{quasar-convex} functions \cite{gower2021sgd,hardt2018gradient,hinder2020near,lara2025delayed,pun2024online}. 
\begin{assumption}\label{ass:quasar_conv}
$f$ is quasar-convex, \textit{i.e.}, there exists $\tau\in(0,1]$ such that for a minimizer $x^\ast \in \argmin f$ and any
$x\in\R^d$
\[f(x) + \frac{1}{\tau}\langle\nabla f(x), x^\ast - x\rangle \le f(x^\ast).\]
\end{assumption}
This assumption appears to model interesting non-convex practical problems, such as learning linear dynamical systems
\cite{hardt2018gradient} and training certain generalized linear models
\cite{wang2023continuizedaccelerationquasarconvex}. From a more empirical perspective, when using stochastic gradient descent, some deep-learning models appear to satisfy $1$-quasar convexity along the optimization trajectory
\cite{sgdQuasConvNeur}, this special case $\tau = 1$ being also known as star-convexity \cite{lee2016optimizingstarconvex}. An important property of quasar-convex functions is that critical points are global minimizers,
ruling out spurious local minima and
saddle points. Remarkably, some acceleration results from the convex setting extend to quasar-convex functions, up to some stabilizing mechanisms.
Precisely, when combined with subspace-search procedures, Nesterov's momentum
methods achieve the accelerated $\bigO(1/\sqrt{\varepsilon})$ rate
\cite{hinder2020near,nesterov2021primal}, matching the bound of the convex case up to a logarithmic factor. The line-search procedure designed in \cite{hinder2020near} has become standard in order to derive momentum-based acceleration under quasar-convexity; see extensions in the case of using stochastic gradients
\cite{fu2023accelerated}, non Euclidian optimization \cite{lezane2024accelerated} and constrained optimization \cite{martinez2025smooth}.

\noindent
\textbf{The challenge of zeroth-order  acceleration with quasar-convex functions \:}
 To our knowledge, quasar-convex optimization has mostly been focused on first-order methods. In contrast, zeroth-order quasar-convex optimization remains unaddressed, with the notable exception of \cite{farzin2025minimisation}. This work shows that a gradient descent-type algorithm with zeroth-order information retains the bound $\bigO(d/\varepsilon)$ of the convex case. However, it is currently unknown whether incorporating Nesterov's momentum can yield a $\bigO(d/\sqrt{\varepsilon})$ bound in this setting.
Importantly, the subspace-search procedures that allow acceleration in quasar-convex optimization crucially relies on gradient information. In Nesterov et al.\cite{nesterov2021primal}, the search procedure is not explicit, but it is assumed to find a parameter that ensures the non-negativity of a quantity involving the gradient of $f$, while the explicit search procedure from \cite{hinder2020near} is designed to satisfy a bound that involves the gradient. It is therefore unclear whether this subspace-search technique, and thus the possibility of acceleration, extends to the zeroth-order setting.

\noindent
\textbf{The continuized Nesterov momentum as an alternative \:}
We consider a less standard alternative gradient-based method, named \textit{continuized Nesterov method}, which has recently been proposed in \cite{even2021continuized}. 
It is based on a process that combines continuous-time momentum dynamics with gradient steps triggered at random times $\tk$. It writes for some sequences of positive parameters $(\gamma_t)$, $(\gamma'_t)$, $(\eta_t)$ and $(\eta'_t)$, as the following Poisson-driven stochastic differential equation 
\begin{equation}\label{eq:cont_origin}
      \left\{\begin{array}{ll}
        \text{d}x_t &= \eta_t(z_t-x_t)\text{d}t - \gamma_t \int_{\Xi}\nabla f(x_{t},\xi) \text{d}N(t,\xi),\vspace{0.1cm} \\
        \text{d}z_t &= \eta'_t(x_t-z_t)\text{d}t - \gamma'_t \int_{\Xi}\nabla f(x_{t},\xi) \text{d}N(t,\xi),
    \end{array}\right.
\end{equation}
where $\text{d}N(t,\xi)= \sum_{k\ge 0} \delta_{(T_k,\xi_k)}(\text{d}t,\text{d}\xi)$ is a Poisson point measure with intensity $\text{d}t \otimes \text{d}\mathcal{P}_\xi$.
The fundamental property of this process is that although it can be analyzed through continuous-time Lyapunov approaches using tools from stochastic calculus theory, it can be evaluated with a computable algorithm that writes as a Nesterov momentum algorithm.
Of special interest for our motivation is that in the case of $L$-smooth quasar convex functions, this system allows to design an algorithm that achieves in expectation the complexity rate $\bigO(1/\sqrt{\varepsilon})$ \textbf{without requiring subspace-search procedures }\cite{wang2023continuizedaccelerationquasarconvex}. This offers a promising alternative to realize our objective. Nonetheless, the literature on this method remains very sparse and, to our knowledge, the only existing works  \cite{even2021continuized,hermant2025continuized,hermant2026continuized,wang2023continuizedaccelerationquasarconvex} focus on first-order optimization. 


\vspace{0.5cm}
\noindent
\textbf{Contributions \:} Our key contribution is to adapt the dynamics \eqref{eq:cont_origin} to design a stochastic zeroth-order continuized Nesterov process, detailed in Section \ref{sec:method}.
In Section \ref{sec:main_results}, we show that when applied to smooth quasar-convex functions, the associated algorithm recovers the complexity from the smooth convex case up to a $\tau$ factor. In particular, if we assume access to stochastic evaluations of functions, we solve \eqref{prob:intro} in at most $\bigO (d/\sqrt{\varepsilon})$ functions evaluations.
Our method is designed such that our final algorithm couples a gradient step with a mirror step.
This allows a better dependence on the dimension in the 1-norm prox setup, particularly when there exists a sparse minimizer or a different scale between coordinates of a solution, as observed in \cite{gorbunov2018accelerated}. 

\section{Zeroth-order stochastic continuized Nesterov}


In our work, we denote $\min_{x\in \R^d} f(x) :=f^\ast$, $x^\ast \in \argmin_{x \in \R^d} f(x)$. For a measurable set $I$, $\mathcal{U}(I)$ denotes the uniform distribution on $I$. We note $\normp{\cdot}$ the usual $\ell_p$ norm for $p\in[1,+\infty]$,
     $\mathbb{B}=\{x\in \R^d \; \lvert \; \norm{x}\leq 1\}$ and $\mathbb{S}=\{x\in \R^d \; \lvert \; \norm{x} = 1\}$.
    For a random variable $X$, a distribution $\mathcal{P}$ of a random variable $\zeta$ and any measurable function $\psi$, we denote $\mathbb{E}_{\zeta}[\psi(X,\zeta)] := \int \psi(X,\zeta)\mathcal{P}(\text{d}\zeta)$, which is the conditional expectation of $\psi(X,\zeta)$ with respect to $X$. For a distribution $\mathcal{P}$, $\mathcal{P}^{\otimes n}$ is the tensor product of $\mathcal{P}$ with itself $n$ times. We define $V$ as a Bregman divergence, which writes for all $x,y\in \R^d$
\begin{equation}\label{eq:def_bregman_div}
    V(x,y) = \proxf(y)-\proxf(x) -\dotprod{\nabla \proxf(x),y-x},
\end{equation}
with a prox-function $\proxf$ assumed to be continuous, differentiable and $1$-strongly convex with respect to $\normp{\cdot}$.

\subsection{Our Method}\label{sec:method}
We design a stochastic continuized zeroth-order Nesterov momentum algorithm, inspired from the first-order case \cite{even2021continuized}. An intuitive way to present the process is as follows: consider a sequence of random times $\tk$ such that $T_0 = 0$ and $T_{k+1}-T_k$ are i.i.d. random variables with exponential distribution $\mathcal{E}(1)$ for all $k \in \N$. For some positive parameters $\eta_t, \gamma_t, \gamma'_t$, we define a continuous-time process $\xz$ on each interval $(T_k,T_{k+1})$ as the solution of the following equation
\begin{align*}\left\{
    \begin{array}{ll}
        \text{d}{x}_t &= \eta_t(z_t-x_t)\text{d}t, \\
        \text{d}{z}_t &= 0.
    \end{array}
\right.
\end{align*}
At $t=T_k$, the process jumps by performing the following steps,
\begin{align*}
    x_{T_k} &= x_{T_k^{-}} - \gamma_{T_k} \gest\Big(x_{T_k^-},v_k,\xiestalg\Big), \vspace{0.1cm}\\ 
    z_{T_k} &= \underset{z\in \R^d}{\argmin} \left\{\gamma'_{T_k}\dotprod{\gest\Big(x_{T_k^{-}},v_k,\xiestalg\Big),z}+ V(z_{T_k^{-}},z) \right\},
\end{align*}
where $v_k\sim  \mathcal{U}(\sphere)$ and $\xi_k^{\{m\}}$ denotes a collection of $m$ independent random vectors $\xi_k^i$, $1\leq i\leq m$ following the same distribution $\mathcal{P}_\xi$. We assume that $\xi_k^{\{m\}}$ and $v_k$ are mutually independent, and $\gest(\cdot)$ is defined as a mini-batch stochastic finite-difference approximation of $\nabla f(x)$ given by
\begin{equation}\label{eq:oracle_g_v_xi}
    \gest\left(x,v,\xi^{\{m\}}\right) := \frac{1}{m}\sum_{i=1}^m \frac{d}{\alpha}\Big(f(x+\alpha v,\xi_i)-f(x,\xi_i)\Big)v,
\end{equation} 
for some smoothing parameter $\alpha> 0$ and batch size $m \in \N^\ast$.
This process can be written in a more compact form as the following stochastic differential equation 
\begin{align}\label{eq:nest_continuizedv2}\tag{ZO-CNM}
\left\{
    \begin{array}{ll}
        \text{d}x_t &= \eta_t(z_t-x_t)\text{d}t - \gamma_t\displaystyle\int_{\sphere\times \Xi^m} \gest\Big(x_{t^-},\vest,\xiest\Big) \text{d}N\left(t,\vest,\xiest\right), \\
        \text{d}z_t &= \displaystyle\int_{\sphere\times \Xi^m} \left[\prox{\gamma'_t\gest\Big(x_{t^-},\vest,\xiest\Big)}{z_{t^-}}-z_{t^-}\right] \text{d}N\left(t,\vest,\xiest\right),
    \end{array}
\right.
\end{align}
with for all vectors $a,b\in \R^d$,
\begin{equation}\label{eq:frame_def_prox_operator}
    \prox{a}{b} = \underset{z\in \R^d}{\argmin} \{ \dotprod{a,z} + V(b,z)\},
\end{equation}
and where $\text{d}N(t,v,\xiest) = \sum_{k\ge 0} \delta_{\bpar{T_k,v_k,\xiestalg}}(\text{d}t,\text{d}v,\text{d}\xiest)$ is a Poisson point measure defined on $\R_{\geq 0}\times\sphere\times \Xi^m$ with intensity $\text{d}t \otimes \text{d}\mathcal{U}(\sphere) \otimes  \text{d}\mathcal{P}_{\xi}^{\otimes m}$. It combines the continuous component through the term $\text{d}t$ and the steps that act at discrete random times through the term $\text{d}N_t$. Our method \eqref{eq:nest_continuizedv2} shares the following fundamental property with the original first-order version: by defining the sequences $\tilde{y}_k := x_{T_{k+1}^-}$, $\tilde{x}_k := x_{T_{k}}$ and $\tilde{z}_k := z_{T_{k}}$, with $k \in \N$, these satisfy a recursive relation that takes the form of a Nesterov momentum algorithm, with stochastic parameters depending on the random times $\tk$. 
\begin{proposition}\label{prop:disc}
Let $\xz$ follow \eqref{eq:nest_continuizedv2} with underlying jump times $\tk$, where we fix $\eta_t=\frac{ \eta}{t}$ for some positive constant $\eta>0$. 
Define $\tilde{\gamma}_{k} :=\gamma_{T_{k+1}^{-}}$, $\tilde{\gamma}'_{k} := \gamma'_{T_{k+1}^{-}}$, $\tilde{y}_k := x_{T_{k+1}^-}$, $\tilde{x}_{k+1} := x_{T_{k+1}}$ and $\tilde{z}_{k+1} := z_{T_{k+1}}$ as evaluations of this process. Then, $(\tilde{y}_k,\tilde{x}_k, \tilde{z}_k)$ writes as a Zeroth-Order NEsterov Momentum (ZO-NEM) algorithm of the form
%
\begin{align*}\label{alg:cont}\tag{ZO-NEM}\left\{
    \begin{array}{ll}
       \yalg{k} &= \bpar{\frac{T_k}{T_{k+1}}}^{\eta} \xalg{k} + \Big(1- \bpar{\frac{T_k}{T_{k+1}}}^{\eta}\Big)\zalg{k}, \\
        \xalg{k+1} &= \yalg{k} - \tilde{\gamma}_k\gest\Big(\yalg{k},\vestalg,\xiestalg\Big), \\
        \zalg{k+1} &= \prox{\tilde{\gamma}'_{k}\gest\Big(\yalg{k},\vestalg,\xiestalg\Big)}{\zalg{k}}.
    \end{array}
\right.
\end{align*}

\end{proposition}
Proposition \ref{prop:disc} generalizes \cite{even2021continuized, wang2023continuizedaccelerationquasarconvex}, using an arbitrary constant $\eta> 0$, a zeroth-order stochastic estimation of the gradient of the function $f$ and a mirror step, see the proof in Appendix \ref{app:disc_proof}.
To summarize the continuized approach, we will further analyze the continuous-time process \eqref{eq:nest_continuizedv2} using an Itô formula. Then, because these dynamics can be evaluated by the algorithm \eqref{alg:cont}, we will be able to transfer the results obtained for our continuous-time analysis to the algorithm. 
\begin{remark}[On the Prox update]\label{rem:prox}
    In the particular case where $\proxf(x) = \frac{1}{2}\norm{x}^2$,
    $\prox{\tilde{\gamma}'_{k}\gest\Big(\yalg{k},\vestalg,\xiestalg\Big)}{\zalg{k}}$ reduces to a gradient step $\zalg{k}- \tilde \gamma'_k \gest\Big(\yalg{k},\vestalg,\xiestalg\Big)$.
Another choice is $\proxf(x) = \frac{e^1 d^{(\kappa-1)(2-\kappa)/\kappa}\log(d)}{2}\normgen{x}_{\kappa}^2$ with $\kappa=1+\frac{1}{\log(d)}$, which makes this prox function $1$-strongly convex with respect to $\normgen{\cdot}_1$, see \cite{ben2001lectures}. Moreover, Gorbunov et al. \cite{gorbunov2018accelerated} observed in the case of zeroth-order convex optimization that this choice can improve the dependence on the dimension in the case of a sparse vector $x_0-x^\ast$, which can happen in some setting such as compressed sensing \cite{candes2006stable,donoho2006compressed}.
The following lemma illustrates how different choices of norm modify the variance of a random vector drawn uniformly on the unit sphere, which will be used in our proof.
\begin{lemma}[Gorbunov et al.\cite{gorbunov2018accelerated}]\label{lem:bound_qnorm}
    Let $v$ be a random vector uniformly distributed on the sphere $\sphere$, $p \in [1,2]$ and $q$ such that $\frac{1}{p}+\frac{1}{q} = 1$. Let $\kappa_d :=  \min\{q-1,16\log(d)-8 \}d^{\frac{2}{q}-1}$. Then, for $d \geq 8$, we have
    \begin{enumerate}[label=(\roman*)]
        \item  $\Ec{\dotprod{v,s}^2\normq{v}^2}{v} \leq \frac{6\kappa_d}{d}\norm{s}^2 \quad $ for all $s \in \R^d$,
        \item $\Ec{\normq{v}^2}{v} \leq \kappa_d$.
    \end{enumerate}
\end{lemma}
\end{remark} 

\subsection{Analysis Tool: Smoothing technique}\label{sec:back_zero}
A useful technique for studying the convergence of zeroth-order algorithms involves a smoothed version of $f$ given by
\[f_\alpha(x) = \Ec{f(x+\alpha u)}{u},\]
where $u$ is a random vector uniformly distributed on the unit ball $\ball$ and $\alpha > 0$ is the smoothing parameter. 
The introduction of this object can be traced back at least to \cite[Eq 9.3.2]{nemirovskij1983problem}.
Its key property is the following.
\begin{lemma}\label{lem:unbiased_prop_g_xi_alpha}
    The random vector $\gest\Big(x,v,\xiest\Big)$ defined in Equation \eqref{eq:oracle_g_v_xi}, is an unbiased estimate of the gradient of the function $f_\alpha$, namely
    \begin{equation}\label{eq:unbiased_prop_g_xi_alpha}
    \Ecbig{\gest\Big(x,v,\xiest\Big)}{v,\xiest}=\nabla f_\alpha(x).
\end{equation}
\end{lemma}
\begin{proof}
We note first that by definition of $f_\alpha$, we have
\begin{equation}\label{eq:general_f_alpha_eq1}
    f_\alpha(x)=\Ec{f_\alpha(x,\xi)}{\xi} \quad\text{ where }\quad
    f_\alpha(x,\xi)=\Ec{f( x+\alpha u,\xi)}{u}.
    \end{equation}
Since the random variables $\xi_i$, $1\leq i\leq m$ have the same distribution $\mathcal{P}_\xi$, we obtain from the definition \eqref{eq:oracle_g_v_xi} that
\begin{equation}\label{eq_lem_unbiased_prop_g_xi_alpha_eq1}
    \Ecbig{\gest\Big(x,v,\xiest\Big)}{v,\xiest}=\Ec{\gest(x,v,\xi)}{v,\xi}.
\end{equation}
Moreover, we have from \cite{flaxman2004online} that
\[ \nabla f_\alpha(x,\xi) = \frac{d}{\alpha}\Ec{f(x + \alpha v,\xi)v}{v \sim \sphere},\]
which leads with the fact that $\Ec{v}{v\sim \sphere} = 0$ to
\begin{equation*}
    \nabla f_\alpha(x,\xi)= \Ec{\frac{d}{\alpha}(f(x + \alpha v,\xi)-f(x,\xi))v}{v \sim \sphere}
\end{equation*}
As a consequence, we have with the definition \eqref{eq:general_f_alpha_eq1} that
\begin{equation}\label{eq_lem_unbiased_prop_g_xi_alpha_eq2}
    \nabla f_\alpha(x)= \Ec{\gest(x,v,\xi)}{v,\xi},
\end{equation}
where $v$ has uniform distribution on $\sphere$. We conclude with \eqref{eq_lem_unbiased_prop_g_xi_alpha_eq1} and \eqref{eq_lem_unbiased_prop_g_xi_alpha_eq2}.
\end{proof}
Lemma \ref{lem:unbiased_prop_g_xi_alpha} indicates that although $g_\alpha$ is constructed based on zeroth-order information, it is an unbiased estimator of $\nabla f_\alpha$. 
Importantly, some properties of $f$--such as quasar convexity--transfer to $f_\alpha$, and moreover, the distance between $f$ and $f_\alpha$ reduces to zero as $\alpha$ goes to zero. Additional assumptions will ensure a satisfying control of the variance of the estimator, and a precise control of the distance between $f$ and $f_\alpha$. 
\begin{assumption}\label{ass:all_l_smooth}
For any random vector $\xi \in \Xi$, there exists $L(\xi)>0$ such that for any $(x,y) \in \R^d \times \R^d$, we have
\[\norm{\nabla f(x,\xi)-\nabla f(y,\xi)} \le L(\xi)\norm{x-y},\] or equivalently 
    \[\vert f(x,\xi) + \dotprod{\nabla f(x,\xi),y-x} - f(y,\xi) \vert\le \frac{L(\xi)}{2}\norm{x-y}^2.\] Also, there exists $L>0$ such that $\sqrt{\Ec{L(\xi)^2}{\xi}}\leq L <+\infty.$
\end{assumption}
Assumption \ref{ass:all_l_smooth} implies that the gradient of $f$ is $L$-Lipschitz continuous.
Even though the estimate $g_\alpha \Big(\cdot,v,\xiest\Big)$ does not require nor differentiability of $f$ nor to have $\nabla f$ is Lipschitz, such smoothness property is necessary to deduce further our accelerated result. Moreover, it allows for the following control (see Lemma \ref{lem:function_lipschitz_transfer}): 
\begin{equation*}
    \vert f(x)-f_\alpha(x) \vert \le \frac{L\alpha^2}{2} \dfrac{d}{d+2}.
\end{equation*}
To our knowledge, this kind of bound exists under the assumption that $f$ is Lipschitz \cite{duchi2012randomized}, or for the gaussian smoothing \cite{nesterov2017random} but in this case the bound involves a factor $d$ due to the higher variance induced by this distribution.
Lastly, a satisfying variance-control will follow from the following assumption \cite{cevher2019linear,vaswani2019fast}.
\begin{assumption}[Strong growth condition]\label{ass:sgc}
There exist positive constants $\sigma^2$ and $\rho\geq 1$ such that for all $x\in \R^d$,
    \begin{equation*}
        \mathbb{E}_{\xi}\left[ \norm{\nabla f(x,\xi)}^2 \right] \leq \rho \norm{\nabla f(x)}^2 + \sigma^2.
    \end{equation*}
    
\end{assumption}
Assumption \ref{ass:sgc} includes as special cases the multiplicative noise setting when $\sigma^2 = 0$, and the bounded variance setting when $\rho=1$. 
In machine learning, the case $\sigma^2 = 0$ is a special instance of the \textit{interpolation} setting \cite{ma2018power}. It allows for the following control of the estimator $g_\alpha$, defined in \eqref{eq:oracle_g_v_xi}.
\begin{restatable}[]{lemma}{gsphereboundq}\label{lem:g_sphere_bound_q}
Suppose that Assumptions \ref{ass:all_l_smooth} and \ref{ass:sgc} are satisfied. Let $q \in [2,+\infty]$ and $\kappa_d :=  \min\{q-1,16\log(d)-1 \}d^{\frac{2}{q}-1}$. Therefore, for all $x\in \R^d$,
\begin{equation*}
    \Ec{\normq{\gest\left(x,\vest,\xiest\right)}^2 }{v,\xiest}
        \leq \dfrac{d\kappa_d}{m}\Bigg[24(\rho+m)\norm{\nabla f_\alpha(x)}^2+25\alpha^2L^2(\rho+m)+24\sigma^2\Bigg]
\end{equation*}
\end{restatable}
\section{Convergence results}\label{sec:main_results}
We first derive a result for our continuous-time dynamics \eqref{eq:nest_continuizedv2}. 

\begin{theorem}\label{thm:acc_nest_zero_v_xi}
Suppose that Assumptions \ref{ass:quasar_conv}, \ref{ass:all_l_smooth}, and \ref{ass:sgc} are satisfied.
Let $\xz$ be the solution of \eqref{eq:nest_continuizedv2} with $\gest(\cdot)$ given in \eqref{eq:oracle_g_v_xi} and the choice of parameters $  \eta_s = \frac{2}{s \tau}, \gamma'_s = \frac{2C s}{\tau}$ and $ \gamma_s = \frac{m}{4dL(\rho+m-1)}$, where $C = \frac{\tau^2m^2}{384d^2L\kappa_d(\rho+m)(\rho+m-1)}$.  Then, we have that 
\begin{equation*}\label{thm:acc_nest_zero_v_xi_res1}
    \begin{aligned}
        &\E{f(x_t)-f^\ast} \\
        &\leq \dfrac{384d^2\kappa_dL}{\tau^2t^2}\left(\dfrac{\rho+m}{m}\right)^2V(x_0,x^\ast)
    +\left[\dfrac{L\alpha^2(d+17)}{192d}+\dfrac{\sigma^2}{12dLm}\right]t+\dfrac{\alpha^2Ld(\tau+2)}{2\tau(d+2)}.
    \end{aligned}
\end{equation*}
\end{theorem}
\begin{proof}[Sketch of Proof]
    The proof is based on the following Lyapunov function
\[
    \varphi(t,x_t,z_t) =  a_t(f_\alpha(x_t)-f^\ast) + V(z_t,x^\ast),\]
    with $a_t=Ct^2$.
    Intuitively, this function can be stochastically derivated with respect to the time thanks to an Itô formula. Importantly, previous work using the continuized technique \cite{even2021continuized,wang2023continuizedaccelerationquasarconvex} used an Itô formula that assumes that $\varphi$ is a smooth function. However, in our setting, $V$ is not necessarily differentiable with respect to its first variable. This can be solved by modifying the Itô formula to handle our non-smooth function $\varphi$. The full proof of Theorem \ref{thm:acc_nest_zero_v_xi} is postponed to Section \ref{app:proof_acc_nest_zero_v_xi}.
\end{proof}

While the first term in the bound of
Theorem \ref{thm:acc_nest_zero_v_xi} is of the order $\bigO(1/t^2)$, the second term grows linearly with $t$ and the third is constant with time. However, these two last terms can be arbitrarily reduced by taking $\alpha$ close enough to zero and $m$ large enough. To reach a given precision $\varepsilon > 0$, it then indicates a compromise between the time $t$, the smoothing parameter $\alpha$ and the batch size $m$. Before doing so, we now state a similar result that holds for the algorithm \eqref{alg:cont}.
\begin{theorem}\label{lem:acc_nest_zero_v_xi_with_Tk}
 Let $\{T_i\}_{0\le i \le k}$ random variables such that $T_{i+1} - T_i$ are  i.i.d  of law $\mathcal{E}(1)$ for all $1\leq i \leq k$, with convention $T_0 = 0$. Suppose that Assumptions \ref{ass:quasar_conv}, \ref{ass:all_l_smooth}, \ref{ass:sgc} are verified. Let the iterations of \eqref{alg:cont} with the choice of parameters $\eta_k = \frac{2}{\tau T_k}$, $\gamma'_k = \frac{2C T_k}{\tau}$ and $ \gamma_s = \frac{m}{4dL(\rho+m-1)}$, where $C = \frac{\tau^2m^2}{384d^2L\kappa_d(\rho+m)(\rho+m-1)}$. Then
 \begin{equation*}
    \begin{aligned}
        \E{ T_k^2(f(\xalgt{k})-f^\ast)} 
        &\leq \dfrac{384d^2\kappa_dL}{\tau^2}\left(\dfrac{\rho+m}{m}\right)^2V(x_0,x^\ast)+\dfrac{\alpha^2Ld(\tau+2)}{2\tau(d+2)}k(k+1)\\
    &\qquad+\left[\dfrac{L\alpha^2(d+17)}{192d}+\dfrac{\sigma^2}{12dLm}\right]k(k+1)(k+2).
    \end{aligned}
\end{equation*}
\end{theorem}
\begin{proof}[Sketch of Proof]
Intuitively, the result is deduced by evaluating the continuous time result of Theorem \ref{thm:acc_nest_zero_v_xi} at $t = T_k$. This will me made rigorous by using an optional stopping theorem. Then, it remains to use Proposition \ref{prop:disc}, which states that \eqref{alg:cont} evaluates the continuous time process \eqref{eq:nest_continuizedv2} at $t=T_k$.
\end{proof}
By definition $T_k$ follows the Gamma distribution  $\Gamma(k,1)$, such that $\E{T_k} = k$, which could lead to the approximation $\E{ T_k^2(f(\xalg{k})-f^\ast)} \approx k^2\E{f(\xalg{k})-f^\ast}$. Concentration inequalities allow us to formally apply such an argument with high probability. Then, we can adjust $k$, $\alpha$ and $m$ to achieve a given precision $\varepsilon > 0$, as stated below.

\begin{theorem}\label{thm:concentration_with_Tk}
Suppose that the same assumptions and the choice of parameters as in Theorem \ref{lem:acc_nest_zero_v_xi_with_Tk} hold. Moreover, assume that for a given accuracy $\varepsilon > 0$ with some constants $\delta_0 >1$ and $\delta_1 \in (0,1]$, we fix
\begin{align*}
     m(\varepsilon)&= \ceil{\dfrac{80\sigma^2(\rho+1)}{\tau}\left(\dfrac{\delta_0^3\kappa_d V(x_0,x^\ast)}{\delta_1^6\varepsilon^3L}\right)^{1/2}},\\
      \alpha(\varepsilon) &= \dfrac{1}{3}\left(\dfrac{\varepsilon^3\tau^4\delta_1^6}{(d+17)^2(\rho+1)^2\delta_0^3L^3}\right)^{1/4}\Big(\kappa_d V(x_0,x^\ast)\Big)^{-1/4}.
\end{align*}
Then, as long as
\begin{equation*}
    k(\varepsilon)= \ceil{\dfrac{40d(\rho+1)}{\delta_1\tau}\left(\dfrac{\delta_0\kappa_d LV(x_0,x^\ast)}{\varepsilon}\right)^{1/2}},
\end{equation*}
we ensure that with probability at least $1-\frac{1}{\delta_0}-e^{-(\delta_1-1-\log(\delta_1))k(\varepsilon)}$
\begin{equation*}
    f\left(\xalgt{k(\varepsilon)}\right)-f^\ast \leq \varepsilon.
\end{equation*}
\end{theorem}

Our result states that with proper choice of bath size and accuracy parameters, we achieve a precision $\varepsilon$ in $\bigO\bpar{\frac{d \sqrt{L} \norm{\xalg{0}-x^\ast}}{ \tau \sqrt{\varepsilon}}}$ iterations in the Euclidean setting, and in $\bigO\bpar{\frac{\sqrt{d} \sqrt{L V(x_0,x^\ast)}}{ \tau \sqrt{\varepsilon}}}$ iterations in the $1$-norm prox setup (Remark~\ref{rem:prox}), see a summary in Table~\ref{table:summary_tab3}. In the $1$-norm prox setup, it is expected that $V(x_0,x^\ast) \approx \norm{x_0-x^\ast}^2$ if $x_0-x^\ast$ is a sparse vector. In such case, the dependence on $d$ is significantly improved.
\paragraph{Comparison with Farzin et al. \cite{farzin2025minimisation}} The only existing convergence result of a zeroth-order algorithm in the smooth quasar-convex setting we are aware of is \cite{farzin2025minimisation}, which assumes exact evaluations of the functions, namely Assumption \ref{ass:sgc} with $\rho = 1$ and $\sigma = 0$, and is based on a Euclidean gradient step. Their complexity is $\bigO\bpar{\frac{dL \norm{\xalg{0}-x^\ast}^2}{ \tau \varepsilon}}$, while ours is $\bigO\bpar{\frac{d \sqrt{L} \norm{\xalg{0}-x^\ast}}{ \tau \sqrt{\varepsilon}}}$. This is in the spirit of the classical kind of improvement we obtain with Nesterov's momentum in convex first-order optimization \cite{nesterov2004introductory}, up to the $d$ factor, that comes from the zeroth-order information and is unavoidable \cite{duchi2015optimal}, and the $\tau^{-1}$ factor, which comes from quasar-convexity and is unavoidable \cite{hinder2020near}. However, we note that this improved complexity when using momentum comes at the cost of a more restricted bound on the choice of $\alpha$, which is consistent with previous works \cite{gorbunov2019stochastic,nesterov2017random}.

Finally, we highlight that our result may be seen as a generalization of \cite{gorbunov2019stochastic} in two directions: we relax convexity with $\tau$-quasar convexity, and we allow $\rho>1$ in Assumption \ref{ass:sgc}. Our bounds are essentially similar up to the factor $\tau$
However, we note that, in contrast, Gorbunov et al.\cite{gorbunov2019stochastic} allows for a supplementary, potentially adversarial noise in the evaluations of $f$, which we did not consider. Following our discussion in the introduction, we recall that it is unclear whether our result is achievable with the more classical approaches of subspace-search procedures \cite{hinder2020near}. 
\EG{
\begin{table}[h!] 
\centering
\begin{tabular}{|c |c | c|} 
 \hline
  & $p=1$ & $p=2$ \\  
 \hline\hline 
  $k(\varepsilon)$  & $\dfrac{(\rho+1)}{\delta_1\tau} \sqrt{\dfrac{d\log(d)\delta_0L\Theta_1}{\varepsilon}}$ & $\dfrac{(\rho+1)}{\delta_1\tau} \sqrt{\dfrac{d^2\delta_0L\Theta_2}{\varepsilon}}$  \\[1ex]  \hline
 $m(\varepsilon)$  & $\dfrac{\sigma^2(\rho+1)}{\delta_1^3\tau}\sqrt{\dfrac{\log(d)\delta_0^3\Theta_1}{\varepsilon^3dL}}$ & $\dfrac{\sigma^2(\rho+1)}{\delta_1^3\tau}\sqrt{\dfrac{\delta_0^3\Theta_2}{\varepsilon^3L}}$  \\[1ex] \hline
 $\alpha(\varepsilon)$  & $\dfrac{\tau}{\sqrt{\rho+1}}\left(\dfrac{\varepsilon^3\delta_1^6 d}{\delta_0^3L^3(d+17)^2\log(d)\Theta_1}\right)^{1/4}$ & $\dfrac{\tau}{\sqrt{\rho+1}}\left(\dfrac{\varepsilon^3 \delta_1^6}{\delta_0^3L^3(d+17)^2\Theta_2}\right)^{1/4}$  \\ [2ex] 
 \hline \hline
\end{tabular}
\caption{Summary of the values for $k(\varepsilon)$, $m(\varepsilon)$ and $\alpha(\varepsilon)$ which ensures that outputs $\xalg{k}$ of our discrete Nesterov momentum algorithm \eqref{alg:cont}
satisfies $f(\xalg{k(\varepsilon)})-f^\ast \leq \varepsilon$ with probability $1-\delta_0^{-1}-e^{-(\delta_1-1-\log(\delta_1))k(\varepsilon)}$ for the special cases $p = 1$ and $p = 2$. The integer $p$ is the one that makes the Bregman divergence 1-strongly convex with respect to $\normp{\cdot}$. For each case, we suppose that there exists positive constant $\Theta_p$ such that $V(x_0,x^\ast)\leq \Theta_p$. Numerical constants or/and rounding operations for $k(\varepsilon)$, $m(\varepsilon)$ and $\alpha(\varepsilon)$ are
omitted for simplicity.}
\label{table:summary_tab3}
\end{table}
}

\section{Proofs of Convergence}\label{sec:proof_main_results}
In this part, we provide the proofs of our main convergence theorems. Section \ref{app:proof_acc_nest_zero_v_xi} is devoted to the continuous-time result, while Section \ref{sec:proof_main_results_disc} is devoted to the discrete-time result.
\subsection{Continuous-Time: Proof of Theorem \ref{thm:acc_nest_zero_v_xi}}\label{app:proof_acc_nest_zero_v_xi}
Let us denote by $\overline{x}_t= (t,x_t,z_t)$. Therefore, it follows from \eqref{eq:nest_continuizedv2} that
\begin{equation*}
    \text{d}\overline{x}_t = \zeta(\overline{x}_t)\text{d}t + \int_{\sphere \times \Xi^m} G\Big(\overline{x}_t,v,\xiest\Big)\text{d}N\Big(t,v,\xiest\Big),
\end{equation*}
where
\begin{equation}\label{eq:proof_acc_zero_v_xi_eq1a}
    \zeta(\overline{x}_t) = \begin{pmatrix}
    1 \\ \eta_t(z_t-x_t) \\ 0
    \end{pmatrix},\quad G\Big(\overline{x}_t,v,\xiest\Big) = \begin{pmatrix}
    0 \\ -\gamma_t g_{\alpha}\Big(x_t,v,\xiest\Big) \\ \prox{\gamma'_t\gest\Big(x_{t},\vest,\xiest\Big)}{z_t}-z_t
    \end{pmatrix},
\end{equation}
and we recall from \eqref{eq:oracle_g_v_xi} that
\begin{equation*}
    \gest\left(x,v,\xi^{\{m\}}\right) = \frac{1}{m}\sum_{i=1}^m \frac{d}{\alpha}\Big(f(x+\alpha v,\xi_i)-f(x,\xi_i)\Big)v.
\end{equation*} 
Let $\varphi$ and $\varphi_0$ defined as
\begin{equation}\label{eq:proof_acc_zero_v_xi_eq1}
    \varphi(t,x,z) =  a_t(f_\alpha(x)-f^\ast) + V(z,x^\ast)\quad  \text{ and } \quad \varphi_0(t,x,z) =  a_t(f_\alpha(x)-f^\ast)
\end{equation}
where $V$ 
is a Bregman divergence given in \eqref{eq:def_bregman_div} and such that
\begin{equation}
    V(z,x^\ast)=\proxf(x^\ast)-\proxf(z)-\dotprod{\nabla \proxf (z),x^\ast-z}.
\end{equation}
By applying Proposition \ref{prop:sto_calc_abridged_with_jump}, we obtain
\begin{equation}\label{eq:proof_acc_zero_v_xi_eq2}
    \begin{aligned}
        &\varphi(\overline{x}_t)-\varphi(\overline{x}_0)\\
        &=\int_{0}^t \dotprod{\nabla \varphi_0(\overline{x}_s),\zeta(\overline{x}_s)}\text{d}s + \int_{0}^t \Ecbig{\varphi\Big(\overline{x}_s + G\Big(\overline{x}_s,v,\xiest\Big)\Big) - \varphi(\overline{x}_s)}{v,\xiest}\text{d}s + M_t,
    \end{aligned}
\end{equation}
where $\mart$ is a martingale such that $\E{M_t} = 0$ for all $t\in \R_+$. 
We compute $\nabla \varphi_0$:
\begin{equation}\label{eq:proof_acc_zero_v_xi_eq3}
    \frac{\partial \varphi_0}{\partial s}= \frac{\text{d} a_s}{\text{d} s}(f_\alpha(x)-f^\ast), \quad 
    \frac{\partial \varphi_0}{\partial x} = a_s\nabla f_\alpha(x) \quad \text{and } \quad \frac{\partial \varphi_0}{\partial z}=0.
\end{equation}
It follows from \eqref{eq:proof_acc_zero_v_xi_eq1a} and \eqref{eq:proof_acc_zero_v_xi_eq3} that
\begin{equation}\label{eq:proof_acc_zero_v_xi_eq4}
     \dotprod{\nabla \varphi_0(\overline{x}_s),\zeta(\overline{x}_s)}=\frac{\text{d} a_s}{\text{d}s}(f_\alpha(x_s)-f^\ast)+a_s \eta_s \dotprod{ \nabla f_\alpha(x_s) ,z_s-x_s}.
\end{equation}
Then, 
\begin{equation}\label{eq:proof_acc_zero_v_xi_eq4b1}
    \begin{aligned}
    \varphi\Big(\overline{x}_s + G\Big(\overline{x}_s,v,\xiest\Big)\Big) &= \varphi(\overline{x}_s)+a_s \Big[f_\alpha\Big(x_s - \gamma_s g_{\alpha}\Big(x_s,v,\xiest\Big)\Big) - f_\alpha(x_s)\Big]\\
    &+ V\Big(\prox{\gamma'_s\gest\Big(x_{s},\vest,\xiest\Big)}{z_s},x^\ast \Big)- V(z_s,x^\ast ).
\end{aligned}
\end{equation}
Moreover, the following inequality holds
\begin{equation}
    \begin{aligned}\label{eq:proof_acc_zero_v_xi_eq4b2}
        V\Big(\prox{\gamma'_s\gest\Big(x_s,\vest,\xiest\Big)}{z_s},x^\ast \Big) 
        &\leq V(z_s,x^\ast )+ \gamma'_s\dotprod{\gest\Big(x_s,\vest,\xiest\Big),x^\ast-z_s} \\
        &+\dfrac{(\gamma'_s)^2}{2}\normq{\gest\Big(x_s,\vest,\xiest\Big)}^2.
    \end{aligned}
\end{equation}
The proof of this inequality can be found, for example, in \cite[Equation (2.19)]{gorbunov2018accelerated}. 
It relies on a Fenchel-Young inequality and a classical identity of Bregman divergences:
$V(b,a) - V(c,a) = -V(c,b) -\dotprod{\nabla \proxf(b)-\nabla \proxf(c),a-b}$ for all $a,b,c \in \R^d$. 
Consequently, the inequalities \eqref{eq:proof_acc_zero_v_xi_eq4b1} and \eqref{eq:proof_acc_zero_v_xi_eq4b2} lead to 
\begin{align*}
    \varphi\Big(\overline{x}_s + G\Big(\overline{x}_s,v,\xiest\Big)\Big) &\leq \varphi(\overline{x}_s) + a_s \Big[f_\alpha\Big(x_s - \gamma_s g_{\alpha}\Big(x_s,v,\xiest\Big)\Big) - f_\alpha(x_s)\Big] \\
    &+ \gamma'_s\dotprod{\gest\Big(x_s,\vest,\xiest\Big),x^\ast-z_s}
        +\dfrac{(\gamma'_s)^2}{2}\normq{\gest\Big(x_s,\vest,\xiest\Big)}^2.
\end{align*}
Taking the conditional expectation yields
\begin{equation}\label{eq:proof_acc_zero_v_xi_eq5}
    \begin{aligned}
   &\Ecbig{\varphi\Big(\overline{x}_s + G\Big(\overline{x}_s,v,\xiest\Big)\Big)- \varphi(\overline{x}_s)}{v,\xiest}\\
   & \leq a_s\Ecbig{f_\alpha\Big(x_s - \gamma_s g_{\alpha}\Big(x_s,v,\xiest\Big)\Big) - f_\alpha(x_s)}{v,\xiest} \\
   &+\gamma_s'\dotprod{\Ec{\gest\Big(x_s,\vest,\xiest\Big)}{v,\xiest},x^\ast-z_s}+ \dfrac{(\gamma_s')^2}{2}\Ec{\normq{\gest\Big(x_s,\vest,\xiest\Big)}^2}{v,\xiest}. 
\end{aligned}
\end{equation}
Furthermore, the estimate $g_\alpha(\cdot)$ is unbiased, see Lemma \ref{lem:unbiased_prop_g_xi_alpha}, namely
\begin{equation}\label{eq:proof_acc_zero_v_xi_eq6}
    \Ec{\gest\Big(x_s,\vest,\xiest\Big)}{v,\xiest}=\nabla f_{\alpha}(x_s).
\end{equation}
Moreover, from Lemma \ref{lem:descent_lemma_alpha}
\begin{equation}
    \begin{aligned}\label{eq:proof_acc_zero_v_xi_eq7}
        &\Ec{f_\alpha\left(x_s - \gamma_s g_{\alpha}\Big(x_s,v,\xiest\Big)\right)-f_\alpha(x_s)}{v,\xiest} \\
        &\leq \gamma_s\left[\dfrac{2\gamma_s Ld(\rho+m-1)}{m}-1\right]\norm{\nabla f_\alpha(x_s)}^2+\dfrac{\gamma_s^2d\alpha^2L^3}{4m}\Big(dm+8(\rho+m-1)\Big)\\
        &\qquad +\dfrac{\gamma_s^2dL\sigma^2}{m},
    \end{aligned}
\end{equation}
and Lemma \ref{lem:g_sphere_bound_q}
\begin{equation}\label{eq:proof_acc_zero_v_xi_eq8}
    \begin{aligned}
        \Ec{\normq{\gest\left(x_s,\vest,\xiest\right)}^2 }{v,\xiest}
        &\leq \dfrac{24d\kappa_d(\rho+m)}{m}\norm{\nabla f_\alpha(x_s)}^2 + \dfrac{25d\kappa_d\alpha^2L^2(\rho+m)}{m}\\
        &\quad +\dfrac{24d\kappa_d\sigma^2}{m},
    \end{aligned}
\end{equation}
where $\kappa_d =  \min\{q-1,16\log(d)-1 \}d^{\frac{2}{q}-1}$.
Consequently, it follows from the four equations \eqref{eq:proof_acc_zero_v_xi_eq5}, \eqref{eq:proof_acc_zero_v_xi_eq6}, \eqref{eq:proof_acc_zero_v_xi_eq7} and \eqref{eq:proof_acc_zero_v_xi_eq8} that
\begin{equation}
    \begin{aligned}\label{eq:proof_acc_zero_v_xi_eq9}
        &\Ecbig{\varphi\Big(\overline{x}_s + G\Big(\overline{x}_s,v,\xiest\Big)\Big) - \varphi(\overline{x}_s)}{v,\xiest}\\
        &\leq \left[\dfrac{12d\kappa_d(\rho+m)(\gamma'_s)^2}{m}+\dfrac{2\gamma_s^2 a_sLd(\rho+m-1)}{m}-a_s\gamma_s\right]\norm{\nabla f_\alpha(x_s)}^2\\
        &+ \gamma_s'\dotprod{x^\ast-z_s,\nabla f_{\alpha}(x_s)}+\dfrac{\gamma_s^2 a_sd\alpha^2L^3}{4m}\Big[dm+8(\rho+m-1)\Big]\\
        &+\dfrac{25d\kappa_d\alpha^2L^2(\rho+m)(\gamma'_s)^2}{2m}
        + \dfrac{d\sigma^2}{m}\Big[\gamma_s^2 a_sL+12(\gamma'_s)^2\kappa_d\Big]. 
    \end{aligned}
\end{equation}
Then, by putting together the two contributions \eqref{eq:proof_acc_zero_v_xi_eq4} and \eqref{eq:proof_acc_zero_v_xi_eq9}, and writing 
\[\dotprod{\nabla f(x_s),z_s-x_s} = \dotprod{\nabla f(x_s),z_s-x^\ast} + \dotprod{\nabla f(x_s),x^\ast-x_s}, \] 
we deduce that
\begin{equation*}
    \begin{aligned}
        &\Ecbig{\varphi\Big(\overline{x}_s + G\Big(\overline{x}_s,v,\xiest\Big)\Big) - \varphi(\overline{x}_s)}{v,\xiest}+\dotprod{\nabla \varphi_0(\overline{x}_s),\zeta(\overline{x}_s)}\\
        &\leq \frac{\text{d} a_s}{\text{d}s}(f_\alpha(x_s)-f^\ast)+a_s \eta_s \dotprod{ \nabla f_\alpha(x_s) ,x^\ast-x_s}+\Big(a_s\eta_s-\gamma_s'\Big)\dotprod{\nabla f_{\alpha}(x_s),z_s-x^\ast}\\
        &+\left[\dfrac{12d\kappa_d(\rho+m)(\gamma'_s)^2}{m}+\dfrac{2\gamma_s^2 a_sLd(\rho+m-1)}{m}-a_s\gamma_s\right]\norm{\nabla f_\alpha(x_s)}^2\\
        &+\dfrac{\gamma_s^2 a_sd\alpha^2L^3}{4m}\Big[dm+8(\rho+m-1)\Big]+\dfrac{25d\kappa_d\alpha^2L^2(\rho+m)(\gamma'_s)^2}{2m}\\
        &+ \dfrac{d\sigma^2}{m}\Big[\gamma_s^2 a_sL+12(\gamma'_s)^2\kappa_d\Big]. 
    \end{aligned}
\end{equation*}
From Lemma \ref{lem:sphere:quasar_convex}, we have
    \[  \frac{1}{\tau}\dotprod{\nabla f_\alpha(x),x^\ast-x} \le -\tau(f_\alpha(x)-f^\ast) + \dfrac{d \alpha^2L}{d+2}, \]
which implies 
\begin{equation}
    \begin{aligned}\label{eq:proof_acc_zero_v_xi_eq9b1}
        &\Ecbig{\varphi\Big(\overline{x}_s + G\Big(\overline{x}_s,v,\xiest\Big)\Big) - \varphi(\overline{x}_s)}{v,\xiest}+\dotprod{\nabla \varphi_0(\overline{x}_s),\zeta(\overline{x}_s)}\\
        &\leq \Bigg(\frac{\text{d} a_s}{\text{d}s}-\tau a_s\eta_s\Bigg)\Big(f_\alpha(x_s)-f^\ast\Big)+\Big(a_s\eta_s-\gamma_s'\Big)\dotprod{\nabla f_{\alpha}(x_s),z_s-x^\ast}\\
        &+\left[\dfrac{12d\kappa_d(\rho+m)(\gamma'_s)^2}{m}+\dfrac{2\gamma_s^2 a_sLd(\rho+m-1)}{m}-a_s\gamma_s\right]\norm{\nabla f_\alpha(x_s)}^2+K_s.
    \end{aligned}
\end{equation}
where
\begin{equation}
    \begin{aligned}\label{eq:proof_acc_zero_v_xi_eq9b2}
        K_s&= \dfrac{\gamma_s^2 a_sd\alpha^2L^3}{4m}\Big[dm+8(\rho+m-1)\Big]+\dfrac{25d\kappa_d\alpha^2L^2(\rho+m)(\gamma'_s)^2}{2m} +\frac{\alpha^2dL}{d+2} a_s\eta_s\\
        &+ \dfrac{d\sigma^2}{m}\Big(\gamma_s^2 a_sL+12(\gamma'_s)^2\kappa_d\Big).
    \end{aligned}
\end{equation}
\paragraph{Parameter choices} Our goal is now to find parameters such that all the terms in the right hand side of \eqref{eq:proof_acc_zero_v_xi_eq9b1} that depends on the trajectories are reduced to zero. For $s>0$, we fix $a_s = s^2C$, such that $\frac{\text{d}a_s}{\text{d}s} = 2sC$, for some constant $C>0$ to be fixed later. We then have
\begin{equation}\label{eq:proof_acc_zero_v_xi_eq10}
    \frac{\text{d}a_s}{\text{d}s} - \tau a_s \eta_s = 0 \Leftrightarrow  \eta_s  = \frac{1}{a_s \tau}\frac{\text{d}a_s}{\text{d}s} = \frac{2}{s \tau}.
\end{equation}
Therefore, we also obtain
\begin{equation}\label{eq:proof_acc_zero_v_xi_eq11}
    a_s \eta_s - \gamma'_s = 0 \Leftrightarrow \gamma'_s = a_s \eta_s = \frac{2Cs}{\tau}.
\end{equation}
Next, we want to have
\begin{equation}\label{eq:proof_acc_zero_v_xi_eq12}
   \dfrac{12d\kappa_d(\rho+m)(\gamma'_s)^2}{m}+\dfrac{2\gamma_s^2 a_sLd(\rho+m-1)}{m}-a_s\gamma_s = 0.
\end{equation}
For any $s> 0$, we find that the quantity $ \frac{2\gamma_s^2 a_sLd(\rho+m-1)}{m}-a_s\gamma_s$ has the minimal value $-\frac{a_s m}{8dL(\rho+m-1)}$ reached at $\gamma_s = \frac{m}{4dL(\rho+m-1)}:=\gamma$, which does not depend on $s$.
Hence, by plugging this choice with $a_s = s^2C$ and Equation \eqref{eq:proof_acc_zero_v_xi_eq11} into \eqref{eq:proof_acc_zero_v_xi_eq12}, we get
\begin{equation}\label{eq:proof_acc_zero_v_xi_eq13}
    -\frac{s^2 Cm}{8dL(\rho+m-1)} + \dfrac{48d\kappa_d(\rho+m)C^2s^2}{m\tau^2}= 0 \Leftrightarrow C = \frac{\tau^2m^2}{384d^2L\kappa_d(\rho+m)(\rho+m-1)}.
\end{equation}
Then, with the previous value of the constant $C$ and the choices 
\begin{equation}\label{eq:proof_acc_zero_v_xi_eq14}
    a_s = Cs^2,~  \eta_s = \frac{2}{s \tau},\, \gamma'_s = \frac{2C s}{\tau}, \, \gamma_s = \frac{m}{4dL(\rho+m-1)},
\end{equation}
we deduce that \eqref{eq:proof_acc_zero_v_xi_eq9b1} and \eqref{eq:proof_acc_zero_v_xi_eq9b2} yield
\begin{equation}\label{eq:proof_acc_zero_v_xi_eq15}
    \Ecbig{\varphi\Big(\overline{x}_s + G\Big(\overline{x}_s,v,\xiest\Big)\Big) - \varphi(\overline{x}_s)}{v,\xiest}+\dotprod{\nabla \varphi(\overline{x}_s),\zeta(\overline{x}_s)}
    \leq K'_s,
\end{equation}
and 
\begin{equation}\label{eq:proof_acc_zero_v_xi_eq16}
    K'_s=\dfrac{6d\kappa_d\alpha^2L^2C^2(\rho+m)(d+17)}{\tau^2m}s^2+\dfrac{96d\kappa_dC^2\sigma^2}{\tau^2m}s^2+\dfrac{2\alpha^2LdC}{\tau(d+2)}s.
\end{equation}

\paragraph{Conclusion} As a consequence, it follows from \eqref{eq:proof_acc_zero_v_xi_eq2} that
\begin{equation}
    \begin{aligned}\label{eq:proof_acc_zero_v_xi_eq17}
        \varphi(\overline{x}_t) &\leq \varphi(\overline{x}_0) + \int_{0}^t K'_s \,\text{d}s + M_t\\
        &=\varphi(\overline{x}_0) +\dfrac{2d\kappa_d\alpha^2L^2C^2(\rho+m)(d+17)}{\tau^2m}t^3+\dfrac{32d\kappa_dC^2\sigma^2}{\tau^2m}t^3+\dfrac{\alpha^2LdC}{\tau(d+2)}t^2 + M_t.
    \end{aligned}
\end{equation}
Recalling the definition of the Lyapunov function \eqref{eq:proof_acc_zero_v_xi_eq1} and the choice of $a_t$ in \eqref{eq:proof_acc_zero_v_xi_eq14}, we deduce from \eqref{eq:proof_acc_zero_v_xi_eq17} that
\begin{equation}\label{eq:proof_acc_zero_v_xi_eq18}
    \begin{aligned}
        &t^2(f_\alpha(x_t)-f^\ast) \\
        &\leq \dfrac{1}{C}V(x_0,x^\ast)
    +\dfrac{2d\kappa_d\alpha^2L^2C(\rho+m)(d+17)}{\tau^2m}t^3+\dfrac{32d\kappa_dC\sigma^2}{\tau^2m}t^3+\dfrac{\alpha^2Ld}{\tau(d+2)}t^2 + \dfrac{1}{C}M_t
    \end{aligned}
\end{equation}
The equation \eqref{eq:proof_acc_zero_v_xi_eq18} bounds $f_\alpha$, while we want to bound $f$.
Therefore, we use Lemma \ref{lem:function_lipschitz_transfer} in order to transfer the previous bound to the function $f$, which yields
\begin{equation}\label{eq:proof_acc_zero_v_xi_eq18b}
    \begin{aligned}
        t^2(f(x_t)-f^\ast)
        &\leq \dfrac{1}{C}V(x_0,x^\ast)
    +\dfrac{2d\kappa_d\alpha^2L^2C(\rho+m)(d+17)}{\tau^2m}t^3+\dfrac{32d\kappa_dC\sigma^2}{\tau^2m}t^3\\
    &+\dfrac{\alpha^2Ld(\tau+2)}{2\tau(d+2)}t^2 + \dfrac{1}{C}M_t.
    \end{aligned}
\end{equation}
Replacing the value of $C$ given in \eqref{eq:proof_acc_zero_v_xi_eq13} and using the fact that $\rho\geq 1$, we obtain
\begin{equation}\label{eq:proof_acc_zero_v_xi_eq19}
    \begin{aligned}
        &t^2(f(x_t)-f^\ast)\\
        &\leq \dfrac{384d^2\kappa_dL}{\tau^2}\left(\dfrac{\rho+m}{m}\right)^2V(x_0,x^\ast)
    +\dfrac{L\alpha^2(d+17)}{192d}t^3+\dfrac{\sigma^2}{12dLm}t^3+\dfrac{\alpha^2Ld(\tau+2)}{2\tau(d+2)}t^2 \\
    &+ \frac{384d^2L\kappa_d(\rho+m)(\rho+m-1)}{\tau^2m^2}M_t.
    \end{aligned}
\end{equation}
Since $M_t$ is a martingale, we have that $\E{M_t} = 0$ for all $t \in \R_+$. Finally, by taking the expectation on both sides of \eqref{eq:proof_acc_zero_v_xi_eq19}, we conclude that
\begin{equation}\label{eq:proof_acc_zero_v_xi_eq20}
    \begin{aligned}
        &\E{f(x_t)-f^\ast}\\
        &\leq \dfrac{384d^2\kappa_dL}{\tau^2t^2}\left(\dfrac{\rho+m}{m}\right)^2V(x_0,x^\ast)
    +\left[\dfrac{L\alpha^2(d+17)}{192d}+\dfrac{\sigma^2}{12dLm}\right]t+\dfrac{\alpha^2Ld(\tau+2)}{2\tau(d+2)}.
    \end{aligned}
\end{equation}
which achieves the proof of Theorem \ref{thm:acc_nest_zero_v_xi}.

\subsection{Discrete-Time Results}\label{sec:proof_main_results_disc}
The fundamental tool to transfer the result from continuous-time to discrete-time is the following.
    \begin{theorem}[Stopping theorem]\label{thm:martingal_stopping}
    Let $(\varphi_t)_{t\in \R_+}$ be a non-negative process with càdlàg trajectories, such that it verifies
    $$  \varphi_t \le K_0 + U_t+ M_t, $$
    for some positive random variable $K_0$, some deterministic function $(U_t)_{t \ge 0}$, and some martingale $\mart$ with $\E{M_0} = 0$. Then, for any almost surely finite stopping time $\tau$, one has 
$$ \E{\varphi_\tau} \le \E{K_0} + \E{U_\tau}.$$
\end{theorem}
The proof of Theorem \ref{thm:martingal_stopping} can be found in Section \ref{app:stopping_theorem}. It will be used to prove Theorem \ref{lem:acc_nest_zero_v_xi_with_Tk}.
\subsubsection{Proof of Theorem \ref{lem:acc_nest_zero_v_xi_with_Tk}}
\begin{proof}
    We recall \eqref{eq:proof_acc_zero_v_xi_eq19} from the proof of Theorem \ref{thm:acc_nest_zero_v_xi}    \begin{equation*}
    \begin{aligned}
        t^2(f(x_t)-f^\ast)
        &\leq \dfrac{384d^2\kappa_dL}{\tau^2}\left(\dfrac{\rho+m}{m}\right)^2V(x_0,x^\ast)
    +\dfrac{L\alpha^2(d+17)}{192d}t^3+\dfrac{\sigma^2}{12dLm}t^3\\
    &+\dfrac{\alpha^2Ld(\tau+2)}{2\tau(d+2)}t^2 +M_t, 
    \end{aligned}
\end{equation*}
for some martingale $(M_t)$ that satisfies $\E{M_t} = 0$.
    Then, applying Theorem \ref{thm:martingal_stopping} with the almost surely finite time $T_k$ as stopping time together with Proposition \ref{prop:disc}, 
    yields
\begin{equation}
\begin{aligned}\label{eq:lem_acc_nest_zero_v_xi_with_Tkeq1}
        \E{ T_k^2(f(\xalgt{k})-f^\ast)} 
        &\leq \dfrac{384d^2\kappa_dL}{\tau^2}\left(\dfrac{\rho+m}{m}\right)^2V(x_0,x^\ast)+\dfrac{\alpha^2Ld(\tau+2)}{2\tau(d+2)}\E{T_k^2}\\
    &+\left[\dfrac{L\alpha^2(d+17)}{192d}+\dfrac{\sigma^2}{12dLm}\right]\E{T_k^3}.
    \end{aligned}
\end{equation}
The moments of a random variable $T_k$ following the Gamma distribution $\Gamma(k,1) $ are well known, namely $\E{T_k^p} = \Pi_{i=1}^p (k+i-1)$, for all $p \in \N^\ast$. Hence, we have that
\begin{equation}\label{eq:lem_acc_nest_zero_v_xi_with_Tkeq2}
        \E{T_k^2} = k(k+1)\qquad  \text{and} \qquad \E{T_k^3} = k(k+1)(k+2).
\end{equation}
Consequently, we conclude from \eqref{eq:lem_acc_nest_zero_v_xi_with_Tkeq1} and \eqref{eq:lem_acc_nest_zero_v_xi_with_Tkeq2} that
    \begin{equation}
    \begin{aligned}
        \E{ T_k^2(f(\xalgt{k})-f^\ast)} 
        &\leq \dfrac{384d^2\kappa_dL}{\tau^2}\left(\dfrac{\rho+m}{m}\right)^2V(x_0,x^\ast)+\dfrac{\alpha^2Ld(\tau+2)}{2\tau(d+2)}k(k+1)\\
    &\qquad+\left[\dfrac{L\alpha^2(d+17)}{192d}+\dfrac{\sigma^2}{12dLm}\right]k(k+1)(k+2).
    \end{aligned}
\end{equation}
$ $
\end{proof}
The fact that the random time $T_k$ follows the Gamma distribution allows us to use the following exponential concentration inequality.
\begin{lemma}[Chernov inequality]\label{prop:chernov}
    Let $T_k \sim \Gamma(k,1)$. Then, for any $0<\delta_1 \leq 1$,
    \[
        \P(T_k \le \delta_1k)\le e^{-\Big(\delta_1-1-\log(\delta_1)\Big)k}.
\]
\end{lemma}
Now, we have all necessary tools in order to demonstrate our Theorem \ref{thm:concentration_with_Tk}.
\subsubsection{Proof of Theorem \ref{thm:concentration_with_Tk}}
\begin{proof}
First, with the Markov property, for some $\delta_0 >1$, we have that 
\begin{equation*}
    \P\Bigg(T_k^2\big(f(\xalgt{k})-f^\ast\big) \ge \delta_0\E{ T_k^2\big(f(\xalgt{k})-f^\ast\big)} \Bigg) \le \frac{1}{\delta_0}.
\end{equation*}
It follows from the previous inequality and Lemma \ref{lem:acc_nest_zero_v_xi_with_Tk} that with probability at least $1-\frac{1}{\delta_0}$, we have
\begin{equation}
\begin{aligned}\label{eq:thm_concentration_with_Tk_eq1}
        T_k^2\big(f(\xalgt{k})-f^\ast\big)
        &\leq \dfrac{384d^2\kappa_dL\delta_0}{\tau^2}\left(\dfrac{\rho+m}{m}\right)^2V(x_0,x^\ast)+\dfrac{\delta_0\alpha^2Ld(\tau+2)}{2\tau(d+2)}k(k+1)\\
    &\qquad+\left[\dfrac{\delta_0L\alpha^2(d+17)}{192d}+\dfrac{\delta_0\sigma^2}{12dLm}\right]k(k+1)(k+2).
    \end{aligned}
\end{equation}
Moreover, for some $\delta_1 \in (0,1]$, we have thanks to Lemma \ref{prop:chernov} that with probability at least $1-e^{-\big(\delta_1-1-\log(\delta_1)\big)k}$
\begin{equation}\label{eq:thm_concentration_with_Tk_eq2}
    \delta_1^2k^2\big( f(\xalgt{k})-f^\ast\big) \leq T_k^2\big(f(\xalgt{k})-f^\ast\big). 
\end{equation}
Therefore, by using the union bound, we deduce from the two contributions \eqref{eq:thm_concentration_with_Tk_eq1} and \eqref{eq:thm_concentration_with_Tk_eq2} that with probability at least $1-\frac{1}{\delta_0}-e^{-(\delta_1-1-\log(\delta_1))k}$,
\begin{equation}
\begin{aligned}
        f(\xalgt{k})-f^\ast
        &\leq \dfrac{\delta_0}{\delta_1^2}\dfrac{384d^2\kappa_dL}{\tau^2k^2}\left(\dfrac{\rho+m}{m}\right)^2V(x_0,x^\ast)+\dfrac{\delta_0}{\delta_1^2}\dfrac{\alpha^2Ld(\tau+2)}{2\tau(d+2)}\left(1+\dfrac{1}{k}\right)\\
    &\qquad+\dfrac{\delta_0}{\delta_1^2}\left[\dfrac{L\alpha^2(d+17)}{192d}+\dfrac{\sigma^2}{12dLm}\right]k\left(1+\dfrac{1}{k}\right)\left(1+\dfrac{2}{k}\right),
    \end{aligned}
\end{equation}
which yields since $m\geq 1$ and $k\geq 1$
\begin{equation}\label{eq:thm_concentration_with_Tk_eq4}
    \begin{aligned}
        f(\xalgt{k})-f^\ast
        &\leq \dfrac{400\delta_0d^2(\rho+1)^2\kappa_dL}{\delta_1^2\tau^2k^2}V(x_0,x^\ast)+\dfrac{\alpha^2\delta_0Ld(\tau+2)}{\delta_1^2\tau(d+2)} +\dfrac{\delta_0L\alpha^2(d+17)}{32d\delta_1^2}k\\
        &+\dfrac{\delta_0\sigma^2}{2dLm\delta_1^2}k.
    \end{aligned}
\end{equation}
For a given accuracy $\varepsilon > 0$, we want to obtain
\begin{equation*}
    \dfrac{400\delta_0d^2(\rho+1)^2\kappa_dL}{\delta_1^2\tau^2k^2}V(x_0,x^\ast)+\dfrac{\alpha^2\delta_0Ld(\tau+2)}{\delta_1^2\tau(d+2)} +\dfrac{\delta_0L\alpha^2(d+17)}{32d\delta_1^2}k+\dfrac{\delta_0\sigma^2}{2dLm\delta_1^2}k \leq \varepsilon.
\end{equation*}
In particular, we have that 
\begin{equation}\label{eq:thm_concentration_with_Tk_eq5}
    \dfrac{400\delta_0 d^2(\rho+1)^2\kappa_dL}{\delta_1^2\tau^2k^2}V(x_0,x^\ast)\leq \frac{\varepsilon}{4} \Longleftrightarrow k \geq \frac{40d(\rho+1)}{\delta_1\tau}\sqrt{\dfrac{\delta_0\kappa_d LV(x_0,x^\ast)}{\varepsilon}}.
\end{equation}
Moreover, we also obtain that
\begin{equation}\label{eq:thm_concentration_with_Tk_eq6}
    \dfrac{\delta_0\sigma^2k}{2\delta_1^2 dLm}\leq\dfrac{\varepsilon}{4} \Longleftrightarrow m\geq \dfrac{2\delta_0\sigma^2k}{\delta_1^2 dL\varepsilon},
\end{equation}
which leads when using \eqref{eq:thm_concentration_with_Tk_eq5}
\begin{equation}\label{eq:thm_concentration_with_Tk_eq7}
    m\geq \dfrac{80\sigma^2(\rho+1)}{\tau}\sqrt{\dfrac{\delta_0^3\kappa_d V(x_0,x^\ast)}{\delta_1^6\varepsilon^3L}}.
\end{equation}
Then, by using the fact that $\tau\leq 1$ and $k\geq 1$, we derive the trivial bound
\begin{equation*}
    \dfrac{\alpha^2\delta_0Ld(\tau+2)}{\delta_1^2\tau(d+2)} +\dfrac{\delta_0L\alpha^2(d+17)}{32d\delta_1^2}k\leq \dfrac{97\delta_0L\alpha^2(d+17)}{32\delta_1^2 \tau d}k.
\end{equation*}
Hence, it gives with the inequality \eqref{eq:thm_concentration_with_Tk_eq5} that
\begin{equation}
\begin{aligned}\label{eq:thm_concentration_with_Tk_eq8}
        \dfrac{97\delta_0L\alpha^2(d+17)}{32\delta_1^2 \tau d}k \leq \frac{\varepsilon}{2} &\Longleftrightarrow \alpha \leq \sqrt{\dfrac{16 \delta_1^2\tau d \varepsilon}{97 \delta_0 L(d+17) k}}\\
    &\Longleftrightarrow \alpha \leq \dfrac{1}{3}\left(\dfrac{\varepsilon^3\tau^4\delta_1^6}{(d+17)^2(\rho+1)^2\delta_0^3L^3}\right)^{1/4}\Big(\kappa_d V(x_0,x^\ast)\Big)^{-1/4}. 
    \end{aligned}
\end{equation}
Consequently, we from \eqref{eq:thm_concentration_with_Tk_eq5}, \eqref{eq:thm_concentration_with_Tk_eq7} and \eqref{eq:thm_concentration_with_Tk_eq8} we deduce that with choices of $  k(\varepsilon)$, $m  (\varepsilon)$ and $\alpha   (\varepsilon)$ specificed as in the statement of Theorem \ref{thm:concentration_with_Tk},
we ensure from \eqref{eq:thm_concentration_with_Tk_eq4} that with probability at least $1-\frac{1}{\delta_0}-e^{-(\delta_1-1-\log(\delta_1))k(\varepsilon)}$
\[f\left(\xalgt{k(\varepsilon)}\right)-f^\ast \leq \varepsilon.\]
\end{proof}

\begin{figure}[H]
    \centering

    \begin{minipage}{0.32\linewidth}
        \centering
        \small \textbf{Logistic -- Low Sparsity}
        
        \includegraphics[width=\linewidth]{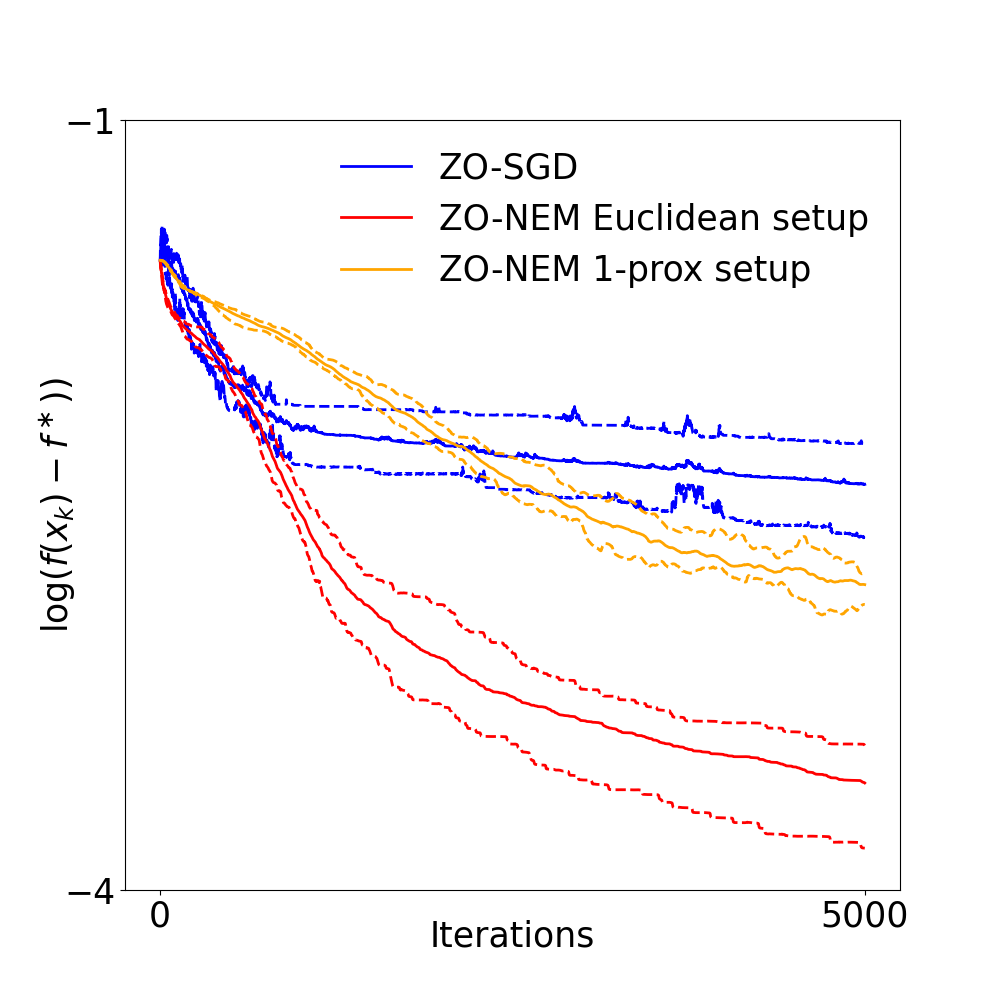}
    \end{minipage}
    \hfill
    \begin{minipage}{0.32\linewidth}
        \centering
        \small \textbf{Quadratic -- Low Sparsity}
        
        \includegraphics[width=\linewidth]{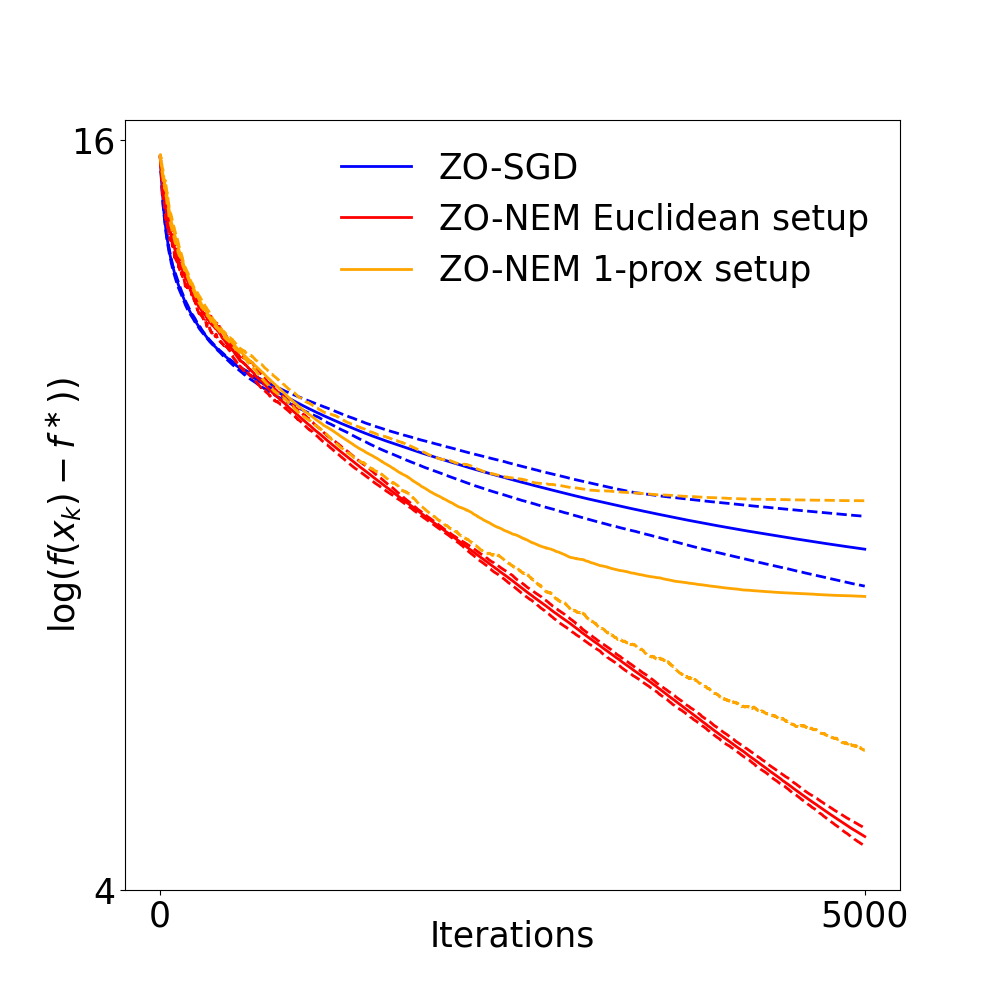}
    \end{minipage}
    \hfill
    \begin{minipage}{0.32\linewidth}
        \centering
        \small \textbf{ReLU -- Low Sparsity}
        
        \includegraphics[width=\linewidth]{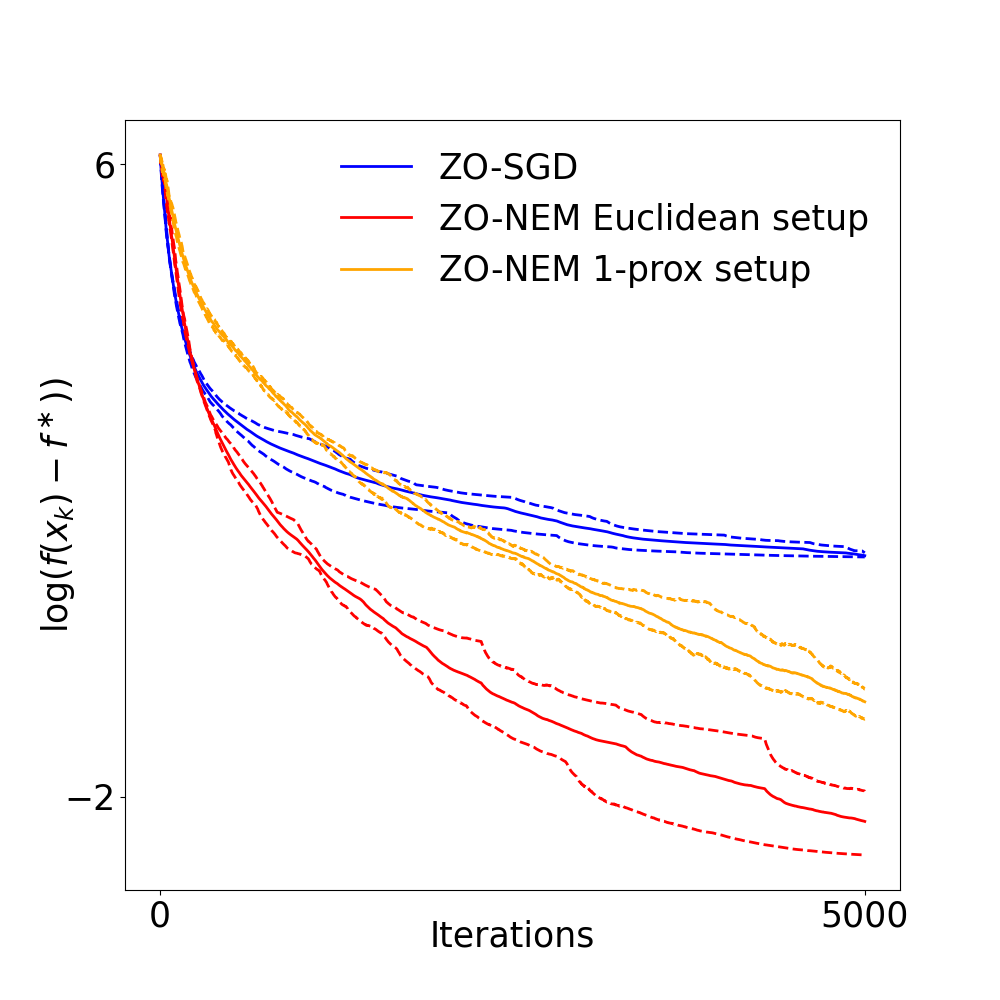}
    \end{minipage}

    \vspace{0.5cm}

    \begin{minipage}{0.32\linewidth}
        \centering
        \small \textbf{Logistic -- High Sparsity}
        
        \includegraphics[width=\linewidth]{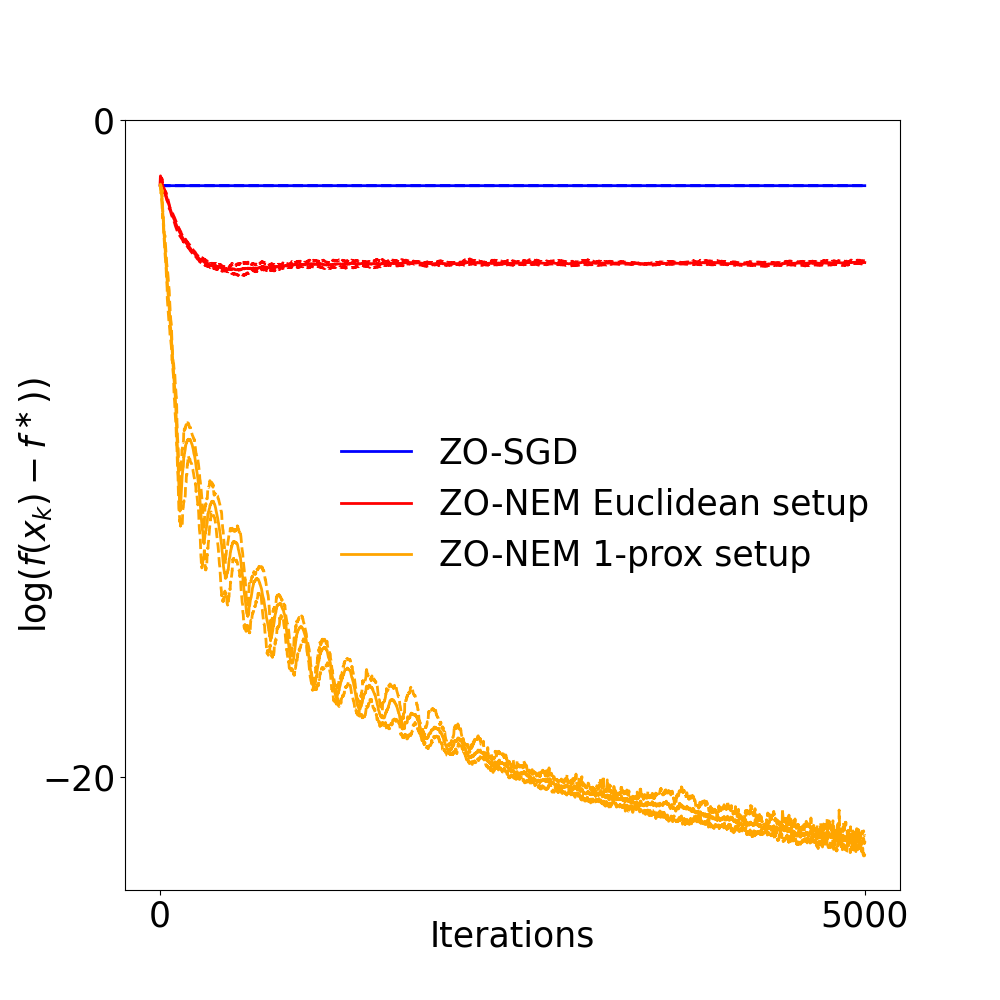}
    \end{minipage}
    \hfill
    \begin{minipage}{0.32\linewidth}
        \centering
        \small \textbf{Quadratic -- High Sparsity}
        
        \includegraphics[width=\linewidth]{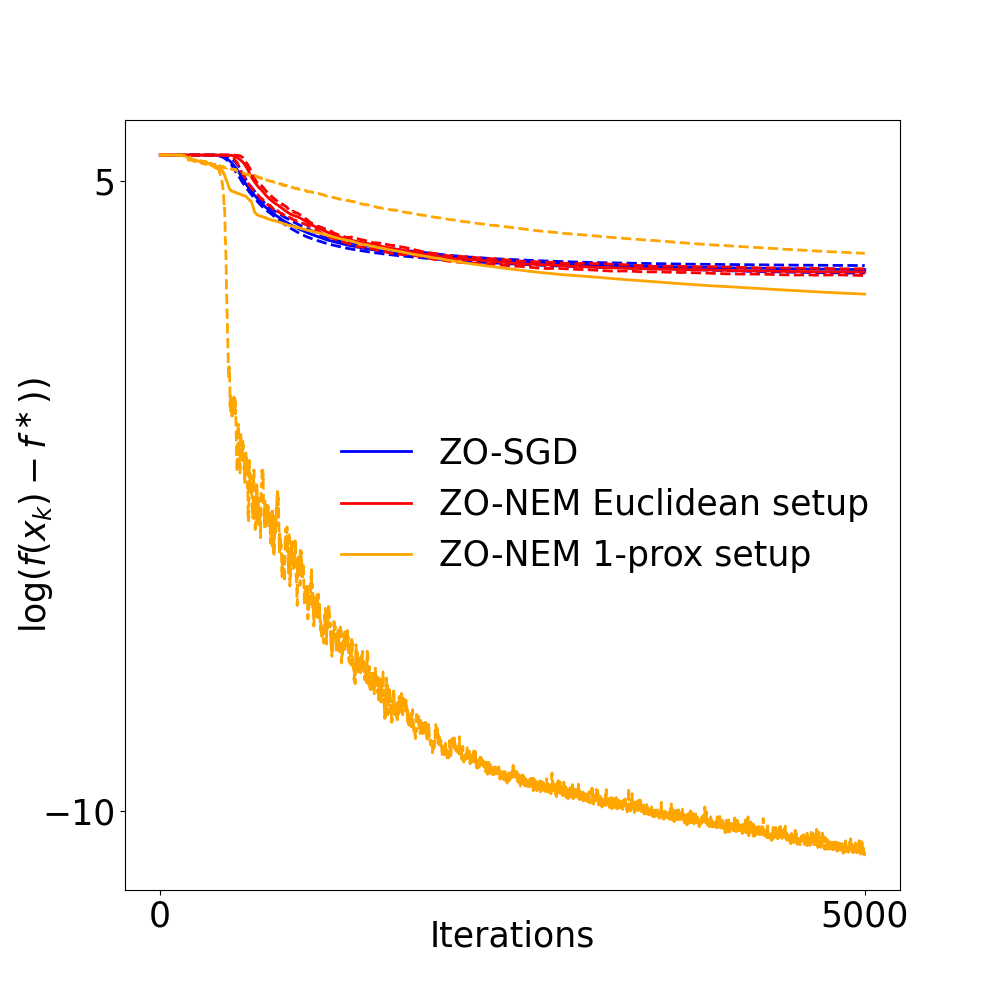}
    \end{minipage}
    \hfill
    \begin{minipage}{0.32\linewidth}
        \centering
        \small \textbf{ReLU -- High Sparsity}
        
        \includegraphics[width=\linewidth]{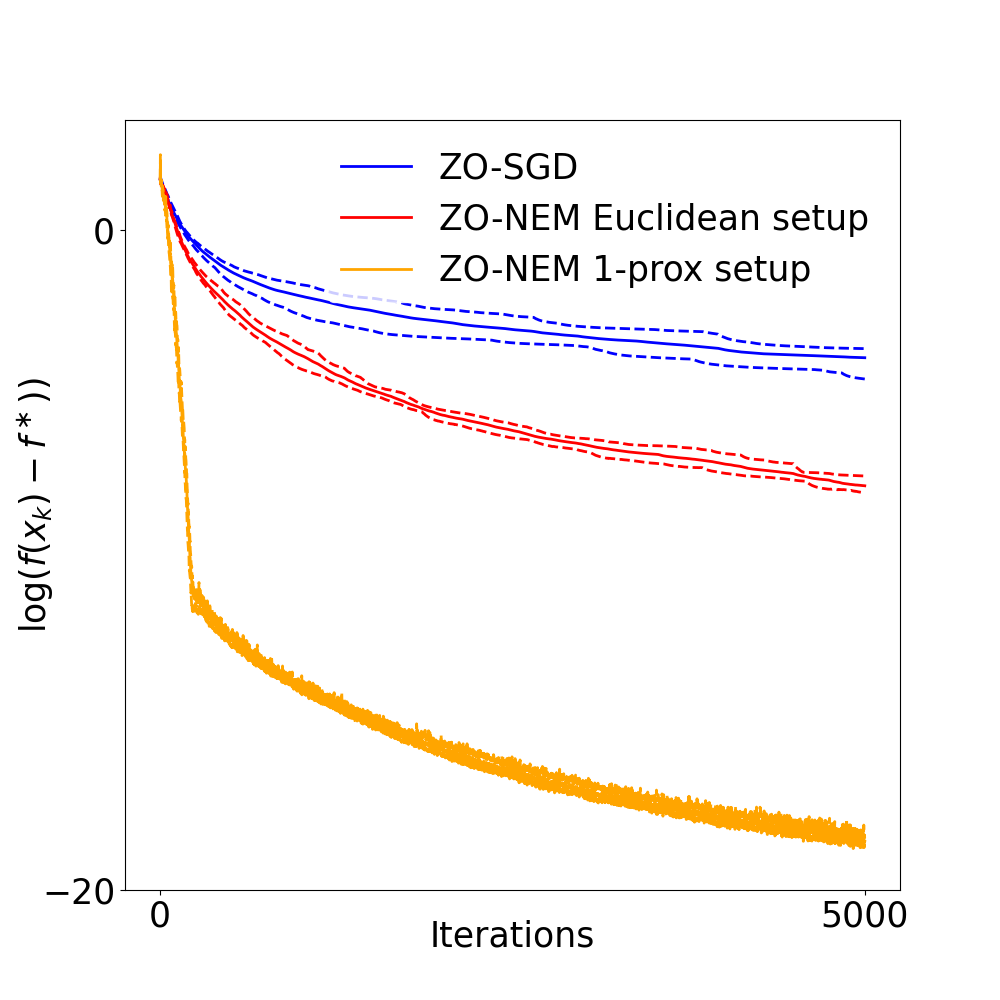}
    \end{minipage}
    \caption{We run \eqref{eq:zo_sgd} and \eqref{alg:cont} in the Euclidean setup and in the $1$-prox setup (see Remark~\ref{rem:prox}), for the training of generalized linear model for different choices of link functions. On the first line, $x^\ast \sim \mathcal{N}(0,I_d)$ which induces no specific sparsity while for the experiment of the second line, $x^\ast$ has all coordinates set to $0$ except the ten first. We observe that momentum improves significantly, and that the $1$-prox setup significantly improves in the sparse regime.} 
    \label{fig:placeholder}
\end{figure}
\section{Numerical experiments}\label{sec:experiments}
In this section, we compare our \ref{alg:cont} algorithm with a zeroth-order stochastic gradient descent method, defined as follows
\begin{equation}\label{eq:zo_sgd}\tag{ZO-SGD}
    \xalg{k+1} = \xalg{k} - \gamma\gest\Big(\xalg{k},\vestalg,\xiestalg\Big).
\end{equation}
This is essentially the algorithm studied in \cite{farzin2025minimisation}. For \eqref{alg:cont}, we run two instances: one with the choice of prox function $\proxf(x) = \frac{1}{2}\norm{x}^2$, which reduces the mirror step to a classic gradient step, and one with the choice $\proxf(x) = \frac{e^1 d^{(\kappa-1)(2-\kappa)/\kappa}\log(d)}{2}\normgen{x}_{\kappa}^2$ with $\kappa=1+\frac{1}{\log(d)}$. In both cases the updates can be explicitly stated and computed with similar computational cost, see \cite{ben2001lectures}.

We consider the training of a Generalized Linear Model (GLM). We generate a sample of size $N=2000$ of $a_i \overset{i.i.d}\sim \mathcal{N}(0,\text{I}_d)$, with $d = 1000$, and for a link function $\sigma : \R \to \R$ and a point $x^\ast \in \R^d$, we define $b_i = \sigma (\dotprod{a_i,x^\ast})$ for each $i \in \{1,\cdots,N\}$. Then, our goal is to approximate $x^\ast$ by minimizing
\[f(x) = \frac{1}{N}\sum_{i=1}^N (\sigma( \dotprod{a_i,x}) - b_i)^2.\]
We consider the logistic link  $\sigma(t) = 1/(1+e^{-t})$, quadratic link $\sigma(t) = t^2$ and Relu link $\sigma(t) = \max(0,t)$. Up to restricting ourselves in a ball around the solution for the quadratic link, the setting thus defined fit our theoretical framework, except for Relu that is not differentiable. Finally, note that we may conveniently create sparsity by selecting $x^\ast$ being a sparse vector. To set ourselves in a low sparsity regime, we may simply select $x^\ast $ following the gaussian distribution $\mathcal{N}(0,\text{I}_d).$ Finally, note that the parameters for each algorithms were tuned via grid-search procedures. 

The results of our experiment are displayed on Figure \ref{fig:placeholder}. For all choices of link functions, there is a clear advantage of \eqref{alg:cont}, except in the initial stage for the quadratic link, which confirms the benefit of using momentum. In the "High Sparsity" regime, the use of the $1$-norm prox shows an undeniable advantage, which again is consistent with the theory.


\appendix

\section{Discretization: Proof of Proposition \ref{prop:disc}}\label{app:disc_proof}
We recall from \eqref{eq:nest_continuizedv2} with the particular choice $\eta_t=\dfrac{\eta}{t}$ for some positive constant $\eta>0$ that for any $t \in (T_k,T_{k+1})$ and for some $k \in \N$, we have
\begin{equation}
    \text{d}x_t = \frac{\eta}{t}(z_t-x_t)\text{d}t \quad \text{and} \quad
        \text{d}z_t = 0.
\end{equation}
This can be explicitly integrated on $(T_k,t]$, which yields

\begin{equation*}
x_t = z_{T_k} + \bpar{\frac{T_k}{t}}^{\eta} (x_{T_k}-z_{T_k})=\bpar{\frac{T_k}{t}}^{\eta} x_{T_k} + \bpar{1-\bpar{\frac{T_k}{t}}^{\eta}}z_{T_k}, \text{ and } z_t = z_{T_k}.
\end{equation*}
Therefore, by taking $t = T_{k+1}^{-}$, it becomes

\begin{equation*}
    \yalg{k} := x_{T_{k+1}^{-}}  =  \bpar{\frac{T_k}{T_{k+1}}}^{\eta} x_{T_k} + \bpar{1-\bpar{\frac{T_k}{T_{k+1}}}^{\eta}}z_{T_k} \quad \text{and} \quad z_{T_{k+1}^-} = z_{T_k}.
\end{equation*}
Now, to get $x_{T_{k+1}}$ and $z_{T_{k+1}}$, it remains to take a gradient step and a mirror step, respectively. Consequently, it follows that 
\begin{equation*}
    \xalg{k+1} := x_{T_{k+1}} = x_{T_{k+1}^-} - \gamma_{T_{k+1^-}}\gest\Big( x_{T_{k+1}^-},\vestalg,\xiestalg\Big) = \yalg{k} - \gamma_{T_{k+1}^-}\gest\Big(\yalg{k},\vestalg,\xiestalg\Big),
\end{equation*}
and 
\begin{equation*}
    \begin{aligned}
        \zalg{k+1} := z_{T_{k+1}} &= z_{T_{k+1}^-} + \left[\prox{\gamma'_{T_{k+1}^-}\gest\Big(x_{T_{k+1}^-},\vestalg,\xiestalg\Big)}{z_{T_{k+1}^-} }-z_{T_{k+1}^-}\right]  \\
        &= \prox{\gamma'_{T_{k+1}^-}\gest\Big(\yalg{k},\vestalg,\xiestalg\Big)}{\zalg{k}}.
    \end{aligned}
\end{equation*}
Finally, defining $\tilde{\gamma}_{k} :=\gamma_{T_{k+1}^{-}}$ and $\tilde{\gamma}'_{k} := \gamma'_{T_{k+1}^{-}}$ concludes the proof.
\section{Technical lemmas}
\subsection{Spherical Smoothing Properties}
\begin{lemma}\label{lem:function_lipschitz_transferv1}
    Assume that the functions $f(\cdot,\xi)$ verify Assumption \ref{ass:all_l_smooth}. Therefore, the functions $f_\alpha(\cdot,\xi)$ also satisfy Assumption \ref{ass:all_l_smooth} with same constant $L(\xi)$.
\end{lemma}
\begin{proof}
We note that $\nabla f_\alpha(x,\xi) = \nabla \Ec{f(x+\alpha u,\xi)}{u\sim \ball} = \Ec{\nabla f(x+\alpha u,\xi)}{u\sim \ball}$, which holds as $f$ is $C^1$ and $\sphere$ is a compact set. Hence, we have for all $x, y \in \R^d$ that
\begin{align*}
      \norm{\nabla f_\alpha(x,\xi)-\nabla f_\alpha(y,\xi)} 
&= \norm{\Ec{\nabla f(x + \alpha u,\xi) - \nabla f(y+\alpha u,\xi)}{u\sim \ball}}\\
&\overset{(i)}\le \Ec{\norm{\nabla f(x + \alpha u,\xi) - \nabla f(y+\alpha u,\xi)}}{u\sim \ball}\\
&\overset{(ii)}\le  L(\xi)\norm{x-y}.
\end{align*}
(i) is due to Jensen's inequality and (ii) by Assumption \ref{ass:all_l_smooth}.
\end{proof}
\begin{lemma}\label{lem:function_lipschitz_transfer}
    Let $f$ satisfies Assumption \ref{ass:all_l_smooth}. Then, for all $x\in \R^d,$ we have 
     \[\vert f(x)-f_\alpha(x) \vert \le \frac{L\alpha^2}{2} \dfrac{d}{d+2}.\]
\end{lemma}
\begin{proof}

We have that for all $x$ in $\R^d$ that
\begin{align}
    | f(x)-f_\alpha(x) |&= \Big\lvert \Ec{f(x,\xi)-f_\alpha(x,\xi)}{\xi}\Big\rvert \nonumber\\
    &\leq \Ecbig{\lvert f(x,\xi)-f_\alpha(x,\xi)\rvert}{\xi}\label{eq:lem_function_lipschitz_transfer_eq1a}\\
    &=\Ecbig{\lvert \Ec{ f(x,\xi)-f(x+\alpha u)}{u\sim \ball}\rvert}{\xi} 
\end{align}
Furthermore, from Assumption \ref{ass:all_l_smooth} it follows
    \begin{equation}\label{eq:lem_function_lipschitz_transfer_eq1}
        \alpha\dotprod{\nabla f(x,\xi),u} - \frac{L(\xi)\alpha^2}{2}\norm{u}^2 \le f(x+\alpha u,\xi) - f(x,\xi) \le \alpha\dotprod{\nabla f(x,\xi),u} + \frac{L(\xi)\alpha^2}{2}\norm{u}^2.
    \end{equation}
    As the uniform distribution over $\ball$ is symmetric, we have
    \begin{equation}\label{eq:lem_function_lipschitz_transfer_eq2}
        \Ec{u}{u \sim \ball}=0,
    \end{equation}
    which immediately gives 
    \begin{equation}\label{eq:lem_function_lipschitz_transfer_eq3}
        \Ec{\dotprod{\nabla f(x,\xi),u}}{u \sim \ball}=0.
    \end{equation}
    Moreover, one can show that  (see for instance Theorem 1.6.8 in \cite{chafai2026phenomenes})
    \begin{equation}\label{eq:lem_function_lipschitz_transfer_eq4}
        \Ec{\norm{u}^2}{u \sim \ball} = \dfrac{d}{d+2}.
    \end{equation}
   Therefore, we deduce from \eqref{eq:lem_function_lipschitz_transfer_eq1} that
    \[ -\frac{L(\xi)\alpha^2}{2} \dfrac{d}{d+2} \le f_\alpha(x,\xi) - f(x,\xi) \le \frac{L(\xi)\alpha^2}{2} \dfrac{d}{d+2},\]
    which in turn implies
    \begin{equation}\label{eq:lem_function_lipschitz_transfer_eq5}
        | f(x,\xi)-f_\alpha(x,\xi) | \le \frac{L(\xi)\alpha^2}{2}\dfrac{d}{d+2}.
    \end{equation}
    Consequently, we deduce from $\Ec{L(\xi)}{\xi} \le \sqrt{\Ec{L(\xi)^2}{\xi} }\le L$, \eqref{eq:lem_function_lipschitz_transfer_eq1a} and \eqref{eq:lem_function_lipschitz_transfer_eq5} that
    \[| f(x)-f_\alpha(x) | \le \frac{L\alpha^2}{2}\dfrac{d}{d+2}.\]
\end{proof}

\begin{lemma}\label{lem:sphere:quasar_convex}
    Let $f$ satisfies Assumptions \ref{ass:all_l_smooth} and \ref{ass:quasar_conv}. Then, we have for all $x\in \R^d$,
    \[f_\alpha(x) + \frac{1}{\tau}\dotprod{\nabla f_\alpha(x),x^\ast-x} \le f^\ast + \frac{\alpha^2L}{\tau} \dfrac{d}{d+2}. \]
\end{lemma}
\begin{proof}
The following proof is inspired from \cite[Lemma 1]{farzin2025minimisation}  who consider access to exact, deterministic function evaluations, the Gaussian smoothing and uses $f_\alpha(x^\ast)$ instead of $f^\ast(:= \min_x f(x))$.
We have
\begin{equation*}
    f_\alpha(x) =\Ecbig{\Ec{f(x+\alpha u,\xi)}{u\sim \ball}}{\xi}
            =\Ecbig{\Ec{f(x+\alpha u,\xi)}{\xi}}{u\sim \ball}
            =\Ec{f(x+\alpha u)}{u\sim \ball}.
\end{equation*}
Hence, it follows 
 \begin{equation}
        \begin{aligned}\label{eq:quas_conv:1}
            f_\alpha(x)
            &\overset{(i)}\le \Ec{f^\ast - \frac{1}{\tau}\dotprod{\nabla f(x+\alpha u),x^\ast -(x + \alpha u)} }{u\sim \ball}\\
            &= f^\ast  - \frac{1}{\tau}\Ec{\dotprod{\nabla f(x + \alpha u),x^\ast - x}}{u\sim \ball} + \frac{1}{\tau}\Ec{\dotprod{\nabla f(x + \alpha u),\alpha u}}{u\sim \ball}\\
            &\overset{(ii)}=f^\ast - \frac{1}{\tau}\dotprod{\nabla f_\alpha(x),x^\ast -x}+ \frac{1}{\tau}\Ec{\dotprod{\nabla f(x + \alpha u),\alpha u}}{u\sim \ball},
        \end{aligned}
    \end{equation}
    where (i) uses Assumption \ref{ass:quasar_conv} and (ii) uses 
    $\Ec{\nabla f(x + \alpha u)}{u\sim \ball} = \nabla f_\alpha(x) $. Then, by Assumption \ref{ass:all_l_smooth}, we obtain that

   \begin{equation}
        \begin{aligned}
           f(x) \le f(x+\alpha u) + \dotprod{\nabla f(x + \alpha u),x - (x+\alpha u)} + \frac{L}{2}\norm{x-(x+\alpha u)}^2,
        \end{aligned}
    \end{equation}
    which we rearrange to deduce
    \begin{equation*}
    \dotprod{\nabla f(x + \alpha u), \alpha u} \le  f(x+\alpha u)-f(x)  + \frac{L\alpha^2}{2}\norm{ u}^2.
    \end{equation*}
Taking the expectation over $u$ on both sides of the previous inequality, it becomes
\begin{equation}
    \begin{aligned}\label{eq:quas_conv:2}
        \Ec{ \dotprod{\nabla f(x + \alpha u),\alpha u}}{u\sim \ball} &\le \Ec{ f(x+\alpha u)-f(x)}{u\sim \ball} + \frac{L\alpha^2}{2}\Ec{\norm{u}^2}{u\sim \ball}\\
        &\overset{(i)}=f_\alpha(x) - f(x) + \frac{L\alpha^2}{2} \dfrac{d}{d+2}\\
        &\overset{(ii)}\le L\alpha^2\dfrac{d}{d+2},
    \end{aligned}
\end{equation}
where $(i)$ uses Equation \eqref{eq:lem_function_lipschitz_transfer_eq4} and $(ii)$ uses Lemma \ref{lem:function_lipschitz_transfer}. We inject \eqref{eq:quas_conv:2} in \eqref{eq:quas_conv:1} and rearrange to obtain the result, namely
\[f_\alpha(x) + \frac{1}{\tau}\dotprod{\nabla f_\alpha(x),x^\ast-x} \le f^\ast + \frac{\alpha^2L}{\tau} \dfrac{d}{d+2}. \]
\end{proof}

The following lemma is Equation (2.35) in \cite{berahas2022theoretical}, fixing in their result $\epsilon := 0$ and $\phi := f$.
\begin{lemma}\label{lem:error_gradient_sphere}
    Suppose that Assumption \ref{ass:all_l_smooth} holds. Then for all $x\in \R^d$,
    \[\norm{\nabla f(x)- \nabla f_\alpha(x)} \le L\alpha. \]
\end{lemma}
\subsection{Variance-Control Bounds}
We now give variance-control bounds. First, Lemma \ref{lem:norm_gradient_xi} consider the first-order estimator $\nabla f(\cdot,\xi)$, from which we deduce bounds for the zeroth order estimator $\gest\Big(\cdot,v,\xiest\Big)$ in Lemma \ref{lem:g_sphere_bound_q} ($p$-norm, $p \in [2,+\infty]$) and in Lemma \ref{lem:g_sphere_bound} ($2$-norm specifically).
\begin{lemma}\label{lem:norm_gradient_xi} Assume that $f$ satisfies Assumptions \ref{ass:all_l_smooth} and \ref{ass:sgc}. Then,
    \[ \Ec{\norm{\nabla f(x,\xi)}^2}{\xi} \leq 2 \rho \norm{\nabla f_\alpha(x)}^2+ 2 \rho \alpha^2 L^2 + \sigma^2.  \]
\end{lemma}
\begin{proof}
    We recall from Assumption \ref{ass:sgc} that there exist positive constants $\sigma^2$ and $\rho\geq 1$ such that for all $x$ in $\R^d$,
    \begin{align*}
        \mathbb{E}_{\xi}\left[ \norm{\nabla f(x,\xi)}^2 \right] &\leq \rho \norm{\nabla f(x)}^2 + \sigma^2 \nonumber \\
        &\leq 2 \rho \norm{\nabla f_\alpha(x)}^2+ 2 \rho \norm{\nabla f_\alpha(x)-\nabla f(x)}^2 + \sigma^2,
    \end{align*} 
    thanks to the elementary inequality $\norm{a+b}^2 \le 2\norm{a}^2 + 2\norm{b}^2$, for $a,b \in \R^d$. 
    We conclude by using Lemma \ref{lem:error_gradient_sphere}, from which we deduce
    \[\mathbb{E}_{\xi}\left[ \norm{\nabla f(x,\xi)}^2 \right] \leq 2 \rho \norm{\nabla f_\alpha(x)}^2+ 2 \rho \alpha^2 L^2 + \sigma^2.\]
\end{proof}
Then, we recall Lemma \ref{lem:g_sphere_bound_q} below before providing the proof.

\begin{lemma*}
    Suppose that Assumptions \ref{ass:all_l_smooth} and \ref{ass:sgc} are satisfied. Let $q \in [2,+\infty]$ and $\kappa_d :=  \min\{q-1,16\log(d)-1 \}d^{\frac{2}{q}-1}$. Therefore, for all $x\in \R^d$,
\begin{equation*}
    \Ec{\normq{\gest\left(x,\vest,\xiest\right)}^2 }{v,\xiest}
        \leq \dfrac{d\kappa_d}{m}\Bigg[24(\rho+m)\norm{\nabla f_\alpha(x)}^2+25\alpha^2L^2(\rho+m)+24\sigma^2\Bigg]
\end{equation*}
\end{lemma*}
\begin{proof}
    We first note that 
\[\normq{\gest\left(x,\vest,\xiest\right)}^2 = \frac{d^2}{\alpha^2 m^2}\bpar{\sum_{i=1}^m (f(x+\alpha v,\xi_i)-f(x,\xi_i))}^2\normq{v}^2.\]
Then, using $\norm{a+b}^2 \le 2\norm{a}^2 + 2\norm{b}^2$ for any vectors $a,b \in \R^d$, we deduce that
\begin{equation}
    \begin{aligned}\label{eq_lem_g_sphere_bound_q_eq1}
       \normq{\gest\left(x,\vest,\xiest\right)}^2&\le  \frac{2d^2}{\alpha^2 m^2}\bpar{\sum_{i=1}^m \Big\lvert f(x+\alpha v,\xi_i)-f(x,\xi_i)- \alpha\dotprod{\nabla f(x,\xi_i),v}\Big\rvert}^2\normq{v}^2 \\
       &\quad+ \frac{2d^2}{\alpha^2 m^2}\bpar{\sum_{i=1}^m \alpha\dotprod{\nabla f(x,\xi_i),v}}^2\normq{v}^2 .
    \end{aligned}
\end{equation}
However, it follows from Assumption \ref{ass:all_l_smooth} applied to $x=x$ and $y = x+\alpha v$ that for any $\xi_i \in \{ \xi_1,\cdots,\xi_m\}$,
\begin{equation}\label{eq_lem_g_sphere_bound_q_eq2}
    \Big\lvert f(x+\alpha v,\xi_i)-f(x,\xi_i)- \alpha\dotprod{\nabla f(x,\xi_i),v} \Big\rvert \le \frac{\alpha^2L(\xi_i)}{2}\norm{v}^2 = \frac{\alpha^2 L(\xi_i)}{2},
\end{equation}
where the equality is due to the fact that $v$ belongs to the unit Euclidean sphere $\sphere$. Moreover, from the Jensen's inequality, we have for any positive constants $a_1,\cdots,a_m$ that the following inequality holds
\begin{equation}\label{eq_lem_g_sphere_bound_q_eq3}
    \left(\sum_{i=1}^m a_i \right)^2\leq m \sum_{i=1}^m a_i^2.
\end{equation}
Therefore, we obtain from \eqref{eq_lem_g_sphere_bound_q_eq2} and \eqref{eq_lem_g_sphere_bound_q_eq3} that
\begin{equation*}
    \bpar{\sum_{i=1}^m \Big\lvert f(x+\alpha v,\xi_i)-f(x,\xi_i)- \alpha\dotprod{\nabla f(x,\xi_i),v}\Big\rvert}^2\normq{v}^2\leq \dfrac{m\alpha^4}{4} \left(\sum_{i=1}^m L(\xi_i)^2\right)\normq{v}^2,
\end{equation*}
which leads by taking the expectation with Assumption \ref{ass:all_l_smooth}, the independence of $\xiest$ and $v$, and because $\sqrt{\Ec{L(\xi)^2}{\xi} }\le L$, to
\begin{equation}\label{eq_lem_g_sphere_bound_q_eq4}
    \begin{aligned}
        &\Ec{ \bpar{\sum_{i=1}^m \Big\lvert f(x+\alpha v,\xi_i)-f(x,\xi_i)- \alpha\dotprod{\nabla f(x,\xi_i),v}\Big\rvert}^2\normq{v}^2}{v,\xiest}\\
        &\leq \dfrac{m^2\alpha^4L^2}{4}\Ec{\normq{v}^2}{v}\leq \dfrac{1}{4}\alpha^2 d^2 L^2\kappa_d,
    \end{aligned}
\end{equation}
where the last step holds thanks to $\Ec{\normq{v}}{v} \le \kappa_d$ given by Lemma \ref{lem:bound_qnorm}-(ii). \\
Furthermore, we have
\begin{equation}\label{eq_lem_g_sphere_bound_q_eq5}
    \bpar{\sum_{i=1}^m \alpha\dotprod{\nabla f(x,\xi_i),v}}^2 = \alpha^2\sum_{i=1}^m \dotprod{\nabla f(x,\xi_i),v}^2 + 2\alpha^2\sum_{i<j}^m \dotprod{\nabla f(x,\xi_i),v}\dotprod{\nabla f(x,\xi_j),v}.
\end{equation}
Taking expectation with respect to $\xiest$, we compute
\begin{align}
       & \Ec{\bpar{\sum_{i=1}^m \alpha\dotprod{\nabla f(x,\xi_i),v}}^2\normq{v}^2 }{\xiest} \nonumber\\ 
       &=\alpha^2 \Ec{\sum_{i=1}^m \dotprod{\nabla f(x,\xi_i),v}^2\normq{v}^2 + 2\sum_{i<j}^m \dotprod{\nabla f(x,\xi_i),v}\dotprod{\nabla f(x,\xi_j),v}\normq{v}^2}{\xiest}\nonumber\\
        &=\alpha^2 \sum_{i=1}^m\Ec{\dotprod{\nabla f(x,\xi_i),v}^2\normq{v}^2}{\xi_i} + 2\alpha^2\sum_{i<j}^m \Ec{\dotprod{\nabla f(x,\xi_i),v}\dotprod{\nabla f(x,\xi_j),v}\normq{v}^2}{\xi_i,\xi_j}\nonumber\\
        &\overset{(i)}{=}\alpha^2m \Ec{\dotprod{\nabla f(x,\xi),v}^2\normq{v}^2}{\xi}  +2\alpha^2\sum_{i<j}^m \Ec{\dotprod{\nabla f(x,\xi_i),v}}{\xi_i}\Ec{\dotprod{\nabla f(x,\xi_j),v}}{\xi_j}\normq{v}^2\nonumber\\
        &=\alpha^2m \Ec{\dotprod{\nabla f(x,\xi),v}^2\normq{v}^2}{\xi} + 2\alpha^2 \sum_{i<j}^m \dotprod{\nabla f(x),v}^2\normq{v}^2\nonumber\\
        &\overset{(ii)}\leq \alpha^2m \Ec{\dotprod{\nabla f(x,\xi),v}^2\normq{v}^2}{\xi} + \alpha^2 m^2 \dotprod{\nabla f(x),v}^2\normq{v}^2.\label{eq_lem_g_sphere_bound_q_eq6}
\end{align}
We use in (i) that $\xi_i$ and $\xi_j$ follow the same law for all $i,j$, and (ii) uses $\sum_{1 \le i<j \le m}1 = \frac{m(m-1)}{2}\le \frac{m^2}{2}.$
We now take expectation with respect to $v$ on \eqref{eq_lem_g_sphere_bound_q_eq6}, and use Lemma \ref{lem:bound_qnorm}-(i), such that
\begin{align}
     &\Ec{\bpar{\sum_{i=1}^m \alpha\dotprod{\nabla f(x,\xi_i),v}}^2\normq{v}^2 }{v,\xiest}  \nonumber\\ 
     &\leq \alpha^2m \Ecbig{\Ec{\dotprod{\nabla f(x,\xi),v}^2\normq{v}^2}{v}}{\xi} + \alpha^2 m^2 \Ec{\dotprod{\nabla f(x),v}^2\normq{v}^2}{v}\nonumber\\
     &\leq \dfrac{6\alpha^2m\kappa_d}{d}\Ec{\norm{\nabla f(x,\xi)}^2}{\xi}+ \dfrac{6\alpha^2m^2\kappa_d}{d}\norm{\nabla f(x)}^2.\label{eq_lem_g_sphere_bound_q_eq7}
\end{align}
Then, we deduce via Lemma \ref{lem:error_gradient_sphere} and Lemma \ref{lem:norm_gradient_xi} that the inequality \eqref{eq_lem_g_sphere_bound_q_eq7} gives
\begin{equation}
    \begin{aligned}\label{eq_lem_g_sphere_bound_q_eq8}
        &\Ec{\bpar{\sum_{i=1}^m \alpha\dotprod{\nabla f(x,\xi_i),v}}^2\normq{v}^2 }{v,\xiest} \\ 
     &\leq \dfrac{12\alpha^2m\kappa_d(\rho+m)}{d}\norm{\nabla f_\alpha(x)}^2 + \dfrac{6\alpha^2m\kappa_d}{d}\Big[2\alpha^2L^2(\rho+m)+\sigma^2\Big].
    \end{aligned}
\end{equation}
Consequently, combining the three contributions \eqref{eq_lem_g_sphere_bound_q_eq1}, \eqref{eq_lem_g_sphere_bound_q_eq4}, and \eqref{eq_lem_g_sphere_bound_q_eq8}, we deduce the following bound
\[\Ec{\normq{\gest\left(x,\vest,\xiest\right)}^2 }{v,\xiest}
        \leq \dfrac{d\kappa_d}{m}\Bigg[24(\rho+m)\norm{\nabla f_\alpha(x)}^2+25\alpha^2L^2(\rho+m)+24\sigma^2\Bigg].\]
$ $
\end{proof}
In the special case $q=2$, one can derive a tighter bound.
\begin{lemma}\label{lem:g_sphere_bound}
Under Assumptions \ref{ass:all_l_smooth} and \ref{ass:sgc}, we have
\begin{equation*}
    \begin{aligned}
        &\Ec{\norm{\gest\left(x,\vest,\xiest\right)}^2}{v,\xiest}\\
        &\leq \dfrac{4d(\rho+m-1)}{m}\norm{\nabla f_\alpha(x)}^2+\dfrac{d\alpha^2L^2}{2m}\Big[dm+8(\rho+m-1)\Big]+\dfrac{2d\sigma^2}{m}.
    \end{aligned}
\end{equation*}
\end{lemma}
The proof of Lemma \ref{lem:g_sphere_bound} is left to the reader. The lines are very similar to that of the general case of $q$. The main difference is that we exploit the exact identity $\E{ v v^T}= (1/d)\I_d$ coming from the fact that $v$ follows the uniform distribution on $\sphere$.
\begin{lemma}\label{lem:descent_lemma_alpha}
Suppose that Assumptions \ref{ass:all_l_smooth} and \ref{ass:sgc} hold. Then, for all $x$ in $\R^d$,
\begin{equation*}
    \begin{aligned}
        &\Ec{f_\alpha\left(x - \gamma g_{\alpha}\Big(x,v,\xiest\Big)\right) - f_\alpha(x)}{v,\xiest}\\
        &\leq \gamma\left[\dfrac{2\gamma Ld(\rho+m-1)}{m}-1\right]\norm{\nabla f_\alpha(x)}^2+\dfrac{d\gamma^2\alpha^2L^3}{4m}\Big[dm+8(\rho+m-1)\Big]+\dfrac{dL\gamma^2\sigma^2}{m}.
    \end{aligned}
\end{equation*}
\end{lemma}
\begin{proof}
We deduce from Lemma \ref{lem:function_lipschitz_transferv1} that under Assumption \ref{ass:all_l_smooth}, the function $f_\alpha$ has also Lipschitz gradient with constant $L$. Thus, 
for all $x$ in $\R^d$ we have
\begin{equation*}
    \begin{aligned}
        &f_\alpha\left(x - \gamma g_{\alpha}\Big(x,v,\xiest\Big)\right)  \\
        &\leq f_\alpha(x) -\gamma\dotprod{\nabla f_{\alpha}(x), g_{\alpha}\Big(x,v,\xiest\Big)} + \frac{\gamma^2 L}{2}\norm{g_{\alpha}\Big(x,v,\xiest\Big)}^2,
    \end{aligned}
\end{equation*}
which leads by taking the expectation to
\begin{equation}\label{eq:lem_descent_lemma_alpha_eq1}
    \begin{aligned}
    &\Ec{f_\alpha\left(x - \gamma g_{\alpha}\Big(x,v,\xiest\Big)\right) - f_\alpha(x)}{v,\xiest}\\
    &\leq  -\gamma\dotprod{\nabla f_{\alpha}(x), \Ec{g_{\alpha}\Big(x,v,\xiest\Big)}{v,\xiest}} + \frac{\gamma^2 L}{2}\Ec{\norm{g_{\alpha}\Big(x,v,\xiest\Big)}^2}{v,\xiest}.
\end{aligned}
\end{equation}
We recall the unbiasedness property of $\gest\Big(x,v,\xiest\Big)$ given in \eqref{eq:unbiased_prop_g_xi_alpha},
\begin{equation}\label{eq:lem_descent_lemma_alpha_eq2}
    \Ec{g_{\alpha}\Big(x,v,\xiest\Big)}{v,\xiest}=\nabla f_{\alpha}(x).
\end{equation}
Consequently, we conclude from \eqref{eq:lem_descent_lemma_alpha_eq1} and \eqref{eq:lem_descent_lemma_alpha_eq2} with Lemma \ref{lem:g_sphere_bound} that
\begin{equation*}
    \begin{aligned}
        \Ec{f_\alpha\left(x - \gamma g_{\alpha}\Big(x,v,\xiest\Big)\right)}{v,\xiest}
        &\leq f_\alpha(x)+ \gamma\left[\dfrac{2\gamma Ld(\rho+m-1)}{m}-1\right]\norm{\nabla f_\alpha(x)}^2\\
        &+\dfrac{\gamma^2L}{2}\left[\dfrac{d\alpha^2L^2}{2m}\Big[dm+8(\rho+m-1)\Big]+\dfrac{2d\sigma^2}{m}\right].
    \end{aligned}
\end{equation*}
$ $
\end{proof}

\section{Continuized Tool-Box}

We state the basic tools needed to make derive convergence results with system \eqref{eq:nest_continuizedv2}. 
\subsection{Itô Formula}

In what follows, we establish our Proposition \ref{prop:sto_calc_abridged_with_jump} which extends the usual Itô formula to the case where the Lyapunov function is not differentiable with respect to all its components. We overcome this difficulty by investigating the particular process \eqref{eq_prop_sto_calc_abridged_with_jump_res1} and exploiting the separability of the Lyapunov function as specified in \eqref{eq_prop_sto_calc_abridged_with_jump_res2}. 
\begin{proposition}\label{prop:sto_calc_abridged_with_jump}
Let $\overline{x}_t= (t,x_t,z_t) \in \R_{\geq 0}\times \R^d \times \R^d$ be a solution of 
\begin{equation}\label{eq_prop_sto_calc_abridged_with_jump_res1}
    \text{d}\overline{x}_t = \zeta(\overline{x}_t)\text{d}t + \int_{\sphere \times \Xi^m} G\Big(\overline{x}_{t^-},v,\xiest\Big)\text{d}N\Big(t,v,\xiest\Big),
\end{equation}
where $\zeta$ and $G$ are defined in \eqref{eq:proof_acc_zero_v_xi_eq1a}. Let $\varphi(t,x,z)$ defined by
\begin{equation}\label{eq_prop_sto_calc_abridged_with_jump_res2}
    \varphi(t,x,z) = \varphi_0(t,x,z) + \varphi_1(z),
\end{equation}
where $\varphi_0 :\R_{\geq 0}\times \R^d \times \R^d \to \R$ is $C^1$ and $\varphi_1 : \R^d \to \R$ is measurable. Then,
\begin{equation*}
    \begin{aligned}
        &\varphi(\overline{x}_t)-\varphi(\overline{x}_0) \\
        &= \int_{0}^t \dotprod{\nabla \varphi_0(\overline{x}_s),\zeta(\overline{x}_s)}\text{d}s + \int_0^t  \mathbb{E}_{v,\xiest}\Big[ \varphi\Big(\overline{x}_s + G\Big(\overline{x}_s,v,\xiest\Big)\Big) - \varphi(\overline{x}_s)\Big]\text{d}s + M_t,
    \end{aligned}
\end{equation*}
where $M_t$ is a martingale such that $\E{M_t}=0$ for all $t\geq 0$.
\end{proposition}
\begin{proof}
Since the function $\varphi_0$ is a $C^1$ function, it follows from the usual Itô Formula for smooth Lyapunov applied to $\varphi_0$ (see \textit{e.g.} \cite[Proposition 2]{even2021continuized}) that there exists a martingale $\big(M_t^{(0)}\big)_{t\ge 0}$ with expectation zero such that
\begin{equation}\label{eq_prop_sto_calc_abridged_with_jump_eq1}
    \begin{aligned}
        &\varphi_0(\overline{x}_t)-\varphi_0(\overline{x}_0)-\int_{0}^t \dotprod{\nabla \varphi_0(\overline{x}_s),\zeta(\overline{x}_s)}\text{d}s \\
        &=  \int_{0}^t \Ecbig{\varphi_0\Big(\overline{x}_s + G\Big(\overline{x}_s,v,\xiest\Big)\Big) - \varphi_0(\overline{x}_s)}{v,\xiest}\text{d}s+ M_t^{(0)}.
    \end{aligned}
\end{equation}
Because $\varphi_1$ is not smooth, we cannot use the same result.
However, $z_t$ as the following specific form:
\[\text{d}z_t = \int_{\sphere \times \Xi^m} \Big(\prox{\gamma'_{t^-}\gest\Big(x_{t^-},\vest,\xiest\Big)}{z_{t^-}}-z_{t^-} \Big) \text{d}N\Big(t,\vest,\xiest\Big), \]
which means it consists only of jumps. One can directly deduce a derivation formula from the property of Poisson integrals. Precisely, we can write
\begin{eqnarray}
    \begin{aligned}\label{eq:calc_sto_nondiff_1}
         \varphi_1(z_t) -  \varphi_1(z_0) &= \sum_{k=0}^{+\infty}  \varphi_1(z_{t \wedge T_{k+1}}) - \varphi_1(z_{t \wedge T_{k}}) \\
        &=\underbrace{\sum_{k=0}^{+\infty} \bpar{ \varphi_1(z_{t \wedge T_{k+1}}) - \varphi_1(z_{t \wedge T_{k+1}^-})}}_{\text{(I)}} + \underbrace{\sum_{k=0}^{+\infty}\bpar{ \varphi_1(z_{t \wedge T_{k+1}^-}) - \varphi_1(z_{t \wedge T_{k}})}}_{\text{(II)}}. 
    \end{aligned}
\end{eqnarray}
This decomposition is valid almost surely for any $t\ge 0$ as $T_n \to +\infty$ almost surely.
 Because $(z_t)_{t \geq 0}$ is constant on the interval $[t \wedge T_{k},t \wedge T_{k+1})$, we obtain that (II)$=0$. Furthermore, we have
\begin{eqnarray*}
    \begin{aligned}
&\sum_{k=0}^{+\infty} \bpar{ \varphi_1(z_{t \wedge T_{k+1}}) - \varphi_1(z_{t \wedge T_{k+1}^-})}
=\sum_{k\ge0:\,T_{k+1} \le t}\bpar{ \varphi_1(z_{t \wedge T_{k+1}}) - \varphi_1(z_{t \wedge T_{k+1}^-})}\\
&=\sum_{k\ge0:\,T_{k+1}\le t} \Big[\varphi_1\Big(z_{T_{k+1}^-} + G_z\Big(\overline{x}_{T_{k+1}^-}, v_{k+1},\xi^{\{m\}}_{k+1}\Big)\Big)-\varphi_1(z_{T_{k+1}^-}) \Big],
\end{aligned}
\end{eqnarray*}
denoting $G_z$ as the component of $G$ relative to $z$, which we recall to be defined in \eqref{eq:proof_acc_zero_v_xi_eq1a}.
Because $\text{d}N(t,v,\xiest) = \sum_{k\ge 0} \delta_{\bpar{T_k,v_k,\xiestalg}}(\text{d}t,\text{d}v,\text{d}\xiest)$, we deduce
\begin{equation}\label{eq:calc_sto_nondiff_2}
   \begin{aligned}
       &\sum_{k=0}^{+\infty} \bpar{ \varphi_1(z_{t \wedge T_{k+1}}) - \varphi_1(z_{t \wedge T_{k+1}^-})} \\
       &=\int_0^t\int_{\sphere\times \Xi^m} \Big[\varphi_1\Big(z_{s-}+G_z\Big(\overline{x}_{s-},v,\xiest\Big)\Big)-\varphi_1(z_{s-})\Big]\,\text{d}N\Big(s,v,\xiest\Big).
   \end{aligned}
\end{equation}
Then, we have
\begin{equation}\label{eq_prop_sto_calc_abridged_with_jump_eq2}
    \varphi_1(z_t) = \varphi_1(z_0) + \int_{0}^t \Ecbig{\varphi_1\Big(z_s + G_z\Big(\overline{x}_s,v,\xiest\Big)\Big) - \varphi_1(z_s)}{v,\xiest}\text{d}s + M_t^{(1)},
\end{equation}
where
\begin{align*}
    M_t^{(1)} &:= \int_0^t\int_{\sphere\times \Xi^m} \Big[\varphi_1\Big(z_{s-}+G_z\Big(\overline{x}_{s-},v,\xiest\Big)\Big)-\varphi_1(z_{s-})\Big]\,\text{d}N\Big(s,v,\xiest\Big)\\
    &-\int_0^t\int_{\sphere\times \Xi^m} \Big[\varphi_1\Big(z_{s-}+G_z\Big(\overline{x}_{s-},v,\xiest\Big)\Big)-\varphi_1(z_{s-})\Big]\,\text{d}s \otimes \text{d}\mathcal{U}(\sphere) \otimes  \text{d}\mathcal{P}_{\xi}^{\otimes m}.
\end{align*}
is a martingale with expectation zero.
Consequently, by putting together the two contributions \eqref{eq_prop_sto_calc_abridged_with_jump_eq1} and \eqref{eq_prop_sto_calc_abridged_with_jump_eq2}, we obtain the desired result 
where $M_t= M_t^{(0)}+M_t^{(1)}$ is a martingale satisfying $\E{M_t} = 0$ for all $t \geq 0$.
\end{proof}

\subsection{Proof of Stopping Theorem \ref{thm:martingal_stopping}}\label{app:stopping_theorem}
The following proof is inspired from \cite[Theorem 6]{hermant2025continuized}.
Let $\tau$ be an almost surely finite time, \textit{i.e.} $ \tau(\omega)< + \infty$ almost surely.
The stopped martingale $\{ M_{t \wedge \tau} \}_{t \in \R_+}$ is a martingale, with $\E{M_{t \wedge \tau}} = \E{M_{0 \wedge \tau}}  = 0$ for all $t \in \R_+$.
    Then, the stopped process $V_t := \varphi_{\tau \wedge t}$ is such that
    \begin{equation*}
        \forall t\ge0,~ \E{V_{t}} \le \E{K_0} + \E{U_{t \wedge \tau}}.
    \end{equation*}
    As we assume $\inf_t \varphi_t \ge 0$, we also have $\inf_t  V_t \ge 0$. We thus apply Fatou's Lemma
    \begin{equation}\label{eq:stopping_thm_0}
        \E{  \liminf_{t \to +\infty} \:V_t}  \le \liminf_{t \to +\infty} \: \E{V_t} \le  \liminf_{t \to +\infty} \:   \{ \E{K_0} + \E{U_{t \wedge \tau}}\},
    \end{equation}
    and also
    \begin{equation}\label{eq:stopping_thm_1}
         \liminf_{t \to +\infty}  \: \{ \E{K_0} + \E{U_{t \wedge \tau}}\} \le \limsup_{t \to +\infty} \: \{ \E{K_0} + \E{U_{t \wedge \tau}}\} \le \E{K_0} + \E{\limsup_{t \to +\infty} \: U_{t\wedge \tau}}.
    \end{equation}
Combining \eqref{eq:stopping_thm_0} and \eqref{eq:stopping_thm_1} yields
\begin{equation}\label{eq:stopping_thm_2}
     \E{  \liminf_{t \to +\infty} \:V_t} \le \E{K_0} + \E{\limsup_{t \to +\infty} \: U_{t\wedge \tau}}.
\end{equation}
    $\tau$ being almost surely finite implies that $\lim_{t \to +\infty} t \wedge \tau(\omega) = \tau(\omega) $, almost surely. Then, for almost all $\omega \in \Omega$, we have
    \begin{equation}\label{eq:stopping_thm_2}
        \lim_{t \to +\infty} V_t(\omega) =  \lim_{t \to +\infty} \varphi_{t\wedge \tau(\omega)}(\omega) = \varphi_{\tau(\omega)}(\omega), 
    \end{equation}
    namely $V_t$ converges almost surely to $\varphi_{\tau}$. Similarly, we have that $U_{t \wedge \tau}$ converges almost surely to $U_\tau$. So, we can rewrite \eqref{eq:stopping_thm_2} as
    \[ \E{\varphi_\tau} \le \E{K_0} + \E{U_\tau}.\]


\bibliographystyle{siamplain}

\end{document}